\documentclass[a4paper,oneside]{amsart}
\usepackage[utf8]{inputenc}
\usepackage[normalem]{ulem}
\usepackage{amsmath}
\usepackage{amssymb}
\usepackage{amsthm}
\usepackage{amscd,mathtools,geometry,textcomp,xcolor,enumerate,graphicx,subfig}
\usepackage{tikz-cd}
\usepackage{tikz}
\usepackage{svg}
\usepackage[all]{xy}
\usepackage{pgfplots}
\numberwithin{equation}{section}

\usepackage{esint,mathpazo,hyperref}

\theoremstyle{plain}
\newtheorem{theorem}{Theorem}[section]
\newtheorem{remark}[theorem]{Remark}
\newtheorem{lemma}[theorem]{Lemma}
\newtheorem{corollary}[theorem]{Corollary}

\newtheorem{proposition}[theorem]{Proposition}

\theoremstyle{definition}
\newtheorem{definition}[theorem]{Definition}
\newtheorem{exampl}[theorem]{Example}

\newcommand{\im}{\mathop{\rm Im}\nolimits}

\newcommand{\vol}{\mathop{\rm vol}\nolimits}

\newcommand{\tr}{\mathop{\rm Tr}\nolimits}
\newcommand{\const}{\mathop{\rm const }\nolimits}

\newcommand{\dive}{\mathop{\rm div}\nolimits}
\DeclareMathOperator{\hess}{Hess}

\newcommand\restr[2]{{% we make the whole thing an ordinary symbol
\left.\kern-\nulldelimiterspace % automatically resize the bar with \right
#1 % the function
\vphantom{\big|} % pretend it's a little taller at normal size
\right|_{#2} % this is the delimiter
}}

\begin{document}
\title{Weighted monotonicity theorems and applications to minimal surfaces of \( \mathbb{H}^n \) and  \( S^n \)}
\author{Manh Tien Nguyen}
%\date{\today}
\address{\parbox{\linewidth}{Département de mathématique, Université Libre de Bruxelles,
    Belgique\\ Institut de Mathématiques de Marseille, France}}
\email{nmtien.math@gmail.com}
%\subjclass[2020]{Primary 53A10, 53C43}
\maketitle

\begin{abstract}
  We prove that in a Riemannian manifold \( M \), each function whose Hessian is
  proportional the metric tensor yields a weighted monotonicity theorem.  Such function
  appears in the Euclidean space, the round sphere \( {S}^n \) and the hyperbolic space
  \( \mathbb{H}^n \) as the distance function, the Euclidean coordinates of
  \( \mathbb{R}^{n+1} \) and the Minkowskian coordinates of \( \mathbb{R}^{n,1} \).  Then
  we show that weighted monotonicity theorems can be compared and that in the hyperbolic
  case, this comparison implies three \( SO(n,1) \)-distinct unweighted monotonicity
  theorems. From these, we obtain upper bounds of the Graham--Witten renormalised area of a
  minimal surface in term of its ideal perimeter measured under different metrics of the
  conformal infinity. Other applications include a vanishing result for knot invariants
  coming from counting minimal surfaces of \( \mathbb{H}^n \) and a quantification of how
  antipodal a minimal submanifold of \( S^n \) has to be in term of its volume.
\end{abstract}

\section{Introduction}
\label{sec:org542a23b}
The classical monotonicity theorem dictates how a minimal submanifold of the Euclidean
space distributes its volume between spheres of different radii centred around the same
point. In the present paper, we prove the following three monotonicity results for minimal
submanifolds of the hyperbolic space. Let \( \Sigma \) be a minimal \( k \)-submanifold,
\( O \) be an interior point of \( \mathbb{H}^n \) and denote by \( A(t) \) be the
\( k \)-volume of \( \Sigma_t = \Sigma\cap B_{\cosh t}(O) \), the region of \( \Sigma \)
of distance at most \( \cosh t \) from \( O \). Define the function \( Q_{\delta}(a,t) \)
with parameter \( \delta = -1, 0, \) or \( +1 \) as:
\begin{equation}
\label{eq:Qdelta}
Q_\delta(a,t) := \int_a^t(t^2 +\delta)^{\frac{k}{2} - 1}dt + \frac{1}{ka}(a^2 +\delta)^{\frac{k}{2}}
\end{equation}
\begin{theorem}[Time-monotonicity]
\label{thm:intro-time}
Let  \(d \) be the hyperbolic distance from \( O \) to \( \Sigma \), \( a \) be any number
in the interval \([1,\cosh d]\). Then the function
\[
\Phi_a(\Sigma,t) = \frac{A(t)}{Q_{-1}(a,t)}\quad
\text{is increasing in }  t.
\]
\end{theorem}

When \( a=1 \), Theorem \ref{thm:intro-time} was proved by Anderson \cite[Theorem
1]{Anderson82_CompleteMinimalVarieties}. For any submanifold \( \Sigma \), if
\( \cosh d \geq a_1\geq a_2 \), then the monotonicity of \( \Phi_{a_1} \) implies that of
\( \Phi_{a_2} \). It is therefore optimal to choose \( a \) in Theorem
\ref{thm:intro-time} to be \( a=\cosh d \). In the language of the Comparison Lemma
\ref{lem:comparison}, the monotonicity of \( \Phi_{a_2} \) is \textit{weaker} than that of
\( \Phi_{a_1} \).

There are 2 other versions of Theorem \ref{thm:intro-time}, in which the interior point
\( O \) is replaced by a totally geodesic hyperplane \( H \) and a point \( b \) on the
sphere at infinity \( S_\infty \). We assume that in the compatification of
\( \mathbb{H}^n \) as the unit ball, the closure of \( \Sigma \) does not intersect
\( H \) or contain \( b \). In a half-space model associated to \( b \), this means that
\( \Sigma \) has a maximal height \( x_{\max} < +\infty \).
\begin{theorem}[Space-monotonicity]
\label{thm:intro-space}
Let \( a \) be any number in \( (0,\sinh d]\) and \( A(t) \) be the volume of the region
\( \Sigma_t \) of \( \Sigma \) that is of distance at most \( \sinh t \) from \( H
\). Then the function
\[
 \frac{A(t)}{Q_{+1}(a,t)}\quad
\text{is increasing in } t.
\]
\end{theorem}

\begin{theorem}[Null-monotonicity]
\label{thm:intro-null}
Let \( a \) be any number in \( (0,x_{\max}^{-1}]\) and \( A(t) \) be the volume of the
region \( \Sigma_t \) of \( \Sigma \) with height at least \( \frac{1}{t} \) in a half
space model associated to \( b \). Then the function
\[
\frac{A(t)}{Q_{0}(a,t)}\quad
\text{is increasing in } t.
\]
\end{theorem}

These theorems can be seen as
statements about volume distribution of a minimal submanifold among level sets of the
time, space and null coordinates of \( \mathbb{H}^n \). These are pullbacks of the
coordinate functions of \( \mathbb{R}^{n,1} \) via the hyperboloid model.  These theorems
are obtained in two steps. First, we prove that for any function \(h\) on a Riemannian
manifold \((M,g)\) with
\begin{equation}
\label{eq:intro-hess}
\hess h = Ug,\quad h\in C^2(M), U\in C^0(M)
\end{equation}
there is a \emph{naturally weighted} monotonicity theorem for minimal submanifolds of
\(M\). This is a statement about the weighted volume distribution of the submanifold among
level sets of \( h \). Here the volume is weighted by \(U\) and the density is obtained by
normalising this volume by that of a gradient tube of \( h \) (see
Definitions \ref{def:weight} and \ref{def:density}). The coordinate
functions of \( \mathbb{R}^{n,1} \) and \( \mathbb{R}^n \) pull back to functions on
\( \mathbb{H}^n \) and \( S^n \) that satisfy \eqref{eq:intro-hess}. In the second step,
we prove a comparison lemma which, in the hyperbolic case, gives us the unweighted volume
distribution which are Theorems \ref{thm:intro-time}, \ref{thm:intro-space},
\ref{thm:intro-null}. More generally, these theorems also hold for \emph{tube extensions}
of minimal submanifolds (cf. Definition \ref{def:tube}).

Since \eqref{eq:intro-hess} is local, more examples for \eqref{eq:intro-hess} can be
constructed as quotient of \( M \) by a symmetry group preserving \( h \). Fuchsian
manifolds are obtained this way (cf. Example \ref{ex:xi-fuchsian}) and Theorem
\ref{thm:intro-space} also holds there, given that one replaces \( H \) by the unique
totally geodesic surface.

There are two major differences, if one is to apply the previous argument for the
Euclidean coordinates of \( S^n \). First, all coordinate functions of \( S^n \) are the
same up to isometry whilst a Minkowskian coordinate \( \xi \) of \( \mathbb{H}^n \) falls
into one of the three types depending on the norm of \( d\xi \) in \( \mathbb{R}^{n,1}
\). Secondly, for the Euclidean coordinates of \( S^n \), the natural weight is
\emph{weaker} than the uniform weight. This means that we cannot recover, from the
Comparison Lemma \ref{lem:comparison}, the unweighted volume distribution of a minimal
submanifold. In fact, the unweighted density of the Clifford torus of \( S^3 \) is not
monotone, even in a hemisphere. Despite this, the framework developed here can be used to
prove (cf. Proposition \ref{prop:eps-sphere}) that the further a minimal submanifold of
\( S^n \) is from being antipodal, the larger its volume has to be. Here we quantify
antipodalness by:

\begin{definition}[]
  \label{def:eps-antipodal}
A subset \( X \) of \( S^n \) is \emph{\( \epsilon \)-antipodal} if
the antipodal point of any point in \( X \) is at distance at most
\( \epsilon \) from \( X \).
\end{definition}

It is clear that 0-antipodal is antipodal and that any subset of
\( S^n \) is \( \pi \)-antipodal. The relation between \( \epsilon \) and the volume
\( A \) of a minimal surface is drawn in Figure \ref{fig:3}. Proposition
\ref{prop:eps-sphere} also confirms, in the non-embedded case, a conjecture by Yau on the
second lowest volume of minimal hypersurfaces of \( S^n \), cf. Remark \ref{rem:yau-conj}.

Another application of the new monotonicity theorems is an isoperimetric inequality for
complete minimal surfaces of \( \mathbb{H}^n \). The area of these surfaces is necessarily
infinite and can be renormalised according to Graham and Witten
\cite{Graham.Witten99_ConformalAnomalySubmanifold}. Let \( A(\epsilon) \) be the area of
the part of the surface that is \( \epsilon \)-far\footnote{see Section
  \ref{sec:renorm-isop-ineq} for a more precise formulation.} from the sphere at infinity
\( S_\infty \). They proved that it has the following expansion
\( A(\epsilon) = \frac{L}{\epsilon} + \mathcal{A}_R + O(\epsilon) \), where the
coefficient \(\mathcal{A}_R \), called the \emph{renormalised area}, is an intrinsic
quantity of the surface. We will give upper bounds of \(\mathcal{A}_R \) by
the length of the ideal boundary measured in different metrics in the conformal class at
infinity. These metrics are obtained as the limit of \( \xi^{-2} g_{\mathbb{H}} \) on
\( S_\infty \), which are finite almost everywhere. Depending on the type of \( \xi \),
they are either round metrics, flat metrics, or doubled hyperbolic metrics.

\begin{theorem}[]
\label{thm:upper-AR}
Let \(\Sigma\subset \mathbb{H}^n\) be a minimal surface with ideal boundary \(\gamma\)
in \( S_\infty \). Let \( O \) be an interior point of \(
\mathbb{H}^n \) and \( |\gamma|_{g_O} \) be the length of \( \gamma \) measured in the
round metric \( g_O \) on \( S_\infty \) associated to \( O \). Then
\begin{equation}
\label{eq:upper-AR-a}
 \mathcal{A}_R(\Sigma) + \frac{1}{2}\left(\cosh d + \frac{1}{\cosh d}\right)|\gamma|_{g_O} \leq 0
\end{equation}
where \( d \) is the distance from \( O \) to \( \Sigma \).
\end{theorem}

Theorem \ref{thm:upper-AR} relies on a simple remark (Lemma \ref{lem:less-than-tube})
about the extra information gained from a monotonicity theorem on the tube extension of a
minimal submanifold. It implies in particular that the volume of a minimal submanifold of \( \mathbb{H}^n \) is not larger than
that of the tube competitor. The estimate \eqref{eq:upper-AR-a} follows from substituting
the volume bound given by Theorem \ref{thm:intro-time} to the Graham--Witten expansion.

The volume bound corresponding to Anderson's monotonicity (Theorem
\ref{thm:intro-time} with \( a=1 \)) was obtained by Choe and Gulliver
\cite{Choe.Gulliver92_SharpIsoperimetricInequality} by a different method. Plugging
this into the Graham--Witten expansion, we have the following weaker version of \eqref{eq:upper-AR-a}:
\begin{equation}
\label{eq:upper-AR-0}
 \mathcal{A}_R(\Sigma) + \sup_{\text{round } \tilde g}|\gamma|_{\tilde g} \leq 0
\end{equation}
The estimate \eqref{eq:upper-AR-0} was also independently proved by Jacob Bernstein
\cite{Bernstein21_SharpIsoperimetricProperty}.  To illustrate the extra factor in
\eqref{eq:upper-AR-a}, we test it on a family of rotational minimal annuli \( M_C \) in
\( \mathbb{H}^4 \), indexed by a real number \( C \). Here the rotation is given by
changing the phase of the 2 complex variables of \( \mathbb{C}^2\cong \mathbb{R}^4 \) by
an opposite number. The profile curves of these annuli are drawn in Figure
\ref{fig:min-HS}(a). Their ideal boundary are pairs of round circles that form Hopf links
in \( S_\infty \).  Whilst \eqref{eq:upper-AR-0} only says
\( \mathcal{A}_R\leq -4\pi \), \eqref{eq:upper-AR-a} shows that they tend to
\( -\infty \). Moreover, the perimeter term of \eqref{eq:upper-AR-a} also captures the
decreasing rate of \( \mathcal{A}_R \). These two quantities are drawn in Figure
\ref{fig:2}. The graph was plotted in log scale, in which the two curves are
asymptotically parallel. So the gap between them should be interpreted as a
multiplicative factor between the two quantities, which is approximately 1.198.

Theorems \ref{thm:intro-space} and \ref{thm:intro-null} also give two other upper bounds of
area, and thus two other isoperimetric inequalities:

\begin{theorem}[]
\label{thm:upper-AR-space}
Let \( H \), \( b \), \( \Sigma \) be a totally geodesic plane, a boundary point, and a
complete minimal surfaces satisfying the condition of Theorems \ref{thm:intro-space} and
\ref{thm:intro-null}. We denote by \( |\gamma|_{H} \) and \( |\gamma|_b \) the length of
the boundary curve under the doubled hyperbolic metric and the flat metric associated to
\( H \) and \( b \). Then
\begin{equation}
\label{eq:upper-AR-space}
\mathcal{A}_R(\Sigma) + \frac{1}{2}\left(\sinh d - \frac{1}{\sinh
    d}\right)|\gamma|_{H}\leq 0
\end{equation}
\begin{equation}
\label{eq:upper-AR-null}
\mathcal{A}_R(\Sigma) + \frac{1}{2}\frac{|\gamma|_{b}}{x_{\max}} \leq 0
\end{equation}
Here \( d \) is the distance from \( \Sigma \) to \( H \) and \( x_{\max} \) is the
maximal height of \( \Sigma \) in the half-space model where \( b \) is at infinity.
\end{theorem}

It can be seen either directly from \eqref{eq:upper-AR-a} or from
\eqref{eq:upper-AR-space} and a rescaling argument (cf. Figure \ref{fig:space-est}) that
\(\mathcal{A}_R\leq -2\pi \) for any minimal surface.  The estimate
\eqref{eq:upper-AR-null}, on the other hand, seems to be optimised for a completely
different type of surface than totally geodesic discs. For the discs,
\( \mathcal{A}_R = -2\pi \) whilst the perimeter term in \eqref{eq:upper-AR-null} is
\( +\pi \) and one may hope to drop the constant \( \frac{1}{2} \) there. However, the
family of minimal annuli of \( \mathbb{H}^3 \) found by Mori
\cite{Mori81_MinimalSurfacesRevolution} and whose renormalised area was computed by Krtouš
and Zelnikov \cite{Krtous.Zelnikov14_MinimalSurfacesEntanglement}, show that the constant
\( \frac{1}{2} \) of \eqref{eq:upper-AR-null} cannot be replaced by any number bigger than
\( 0.59 \).

Each \( (k-1) \)-submanifold \( \gamma \) of \( S_\infty \) induces a function
\( V_\gamma \) on \( \mathbb{H}^n \), whose value at a point \( O \) is the volume of
\( \gamma \) under the round metric \( g_O \). The visual hull of \( \gamma \) was defined
by Gromov \cite{Gromov83_FillingRiemannianManifolds} as the set of points where
\( V_\gamma \) is at least the volume of the Euclidean \( (k-1) \)-sphere.  Another
application of Theorem \ref{thm:intro-time} is:
\begin{corollary}[]
  A minimal submanifold of \( \mathbb{H}^n \) is contained in the \textit{visual hull} of
  its ideal boundary.
\end{corollary}

\begin{proposition}
\label{prop:separation}
Let \(L:=  L_1 \sqcup L_2\) be a separated union of two \( (k-1) \)-submanifold \(L_1,
L_2\) of \({S}^{n-1}\). Then we can rearrange \(L\) in its isotopy class such
that there is no connected minimal \( k \)-submanifold \(\Sigma\) of \(\mathbb{H}^n\)
whose ideal boundary is \(L\).
\end{proposition}
\begin{proof}
  In the Poincaré ball model, we isotope \(L\) so that \(L_1\) (or \(L_2\)) is
  contained in a small ball centred at the North (respectively South) pole and so that
  the Euclidean volume of \(L\) is less than \(\frac{1}{2}\omega_{k-1}\).  It suffices to
  prove that any minimal submanifold filling \( L \) has no intersection with the
  equatorial hyperplane. By convexity, such intersection is contained in a small ball
  centred at the origin \(O\). If it was non-empty, by a small Möbius transform we could
  suppose that \(\Sigma\) contains \(O\) while keeping the Euclidean length of \(L\) less
  than \(\omega_{k-1}\).  This is a contradiction.
\end{proof}

The visual hull and the proof of Proposition \ref{prop:separation} can be visualised in
Figure \ref{fig:vishull}.

The motivation for Proposition \ref{prop:separation} comes from recent works of Alexakis,
Mazzeo \cite{Alexakis.Mazzeo10_RenormalizedAreaProperly} and Fine
\cite{Fine21_KnotsMinimalSurfaces}. The counting problem for minimal surfaces of
\( \mathbb{R}^3 \) bounded by a given curve traces back to Tomi--Tromba's resolution of
the embedded Plateau problem \cite{Tomi.Tromba78_ExtremeCurvesBound} and was studied in a
more general context by White \cite{White87_SpaceDimensionalSurfaces}. In the hyperbolic
space, one can ask the ideal boundary of the minimal surface to be an embedded curve
\( \gamma \) in the sphere at infinity. This problem was studied by Alexakis and Mazzeo in
\( \mathbb{H}^3 \) and by Fine in \( \mathbb{H}^4 \). Denote by \( \mathcal{C} \) the
Banach manifold of curves in \( S_\infty \), \( \mathcal{S} \) the space of minimal
surfaces and \( \pi: \mathcal{S} \longrightarrow \mathcal{C} \) the map sending a surface
to its ideal boundary. The counting problem consists of proving \( \mathcal{S} \) is a
Banach manifold and establishing a degree theory for \( \pi \). The degree is constant in
the isotopy class of \( \gamma \) and defines a knot/link invariant. Proposition
\ref{prop:separation} shows that this invariant is multiplicative in \( \gamma
\): any minimal surface filling a separated union of links is union of surfaces filling
each one individually.

It was proved, in the works mentioned above, that the map \( \pi \) is Fredholm and of
index 0. This is because the stability operator is self-adjoint and elliptic or
0-elliptic. The properness of \( \pi \) is however more subtle. In \( \mathbb{H}^4 \), this can be seen via the family
\( M_C \): as \( C \) decreases to \( 0 \), their boundary converges to
the standard Hopf link \( \{zw = 0\}\cap S^3 \) and the waist of \( M_C \) collapses. In
fact, it can be proved (Proposition \ref{prop:minimising-cone}) that the only minimal
surface of \( \mathbb{H}^4 \) filling the standard Hopf link is the pair of
discs. One can still hope that properness holds on a residual subset of
\( \mathcal{C} \), as the previous phenomenon happens in codimension 2 of \( \gamma \).

Since monotonicity theorems are inequalities, they still hold when the Hessian of the
function \(h\) is comparable to the metric as symmetric 2-tensors. We will see in Section
5 that such function arises naturally as the distance function in a Riemannian manifold
whose sectional curvature is bounded from above. When the curvature is negative, this
weighted monotonicity theorem also implies the unweighted version. In the case of positive
curvature, an unweighted monotonicity \emph{inequality} was obtained by Scharrer
\cite{Scharrer21_GeometricInequalitiesVarifolds}.

\vspace{3mm}

\paragraph{\bf Acknowledgements.}
The author thanks Joel Fine for introducing him to minimal surfaces and harmonic maps in
hyperbolic space and for proposing the search of the minimal annuli in \( \mathbb{H}^4 \). He also wants to thank Christian Scharrer for the reference
\cite{Hoffman.Spruck74_SobolevIsoperimetricInequalities} and Benjamin Aslan for helpful discussions on \cite{Berndt.etal16_SubmanifoldsHolonomy} and
\cite{Hsiang.Lawson71_MinimalSubmanifoldsLow}. The author was supported by
\emph{Excellence of Science grant number 30950721, Symplectic techniques in differential
  geometry}.

\begin{figure}%
    \centering
    \subfloat[in \( \mathbb{H}^4 \) with different values of \( C\).]{{\includegraphics[width=6.5cm]{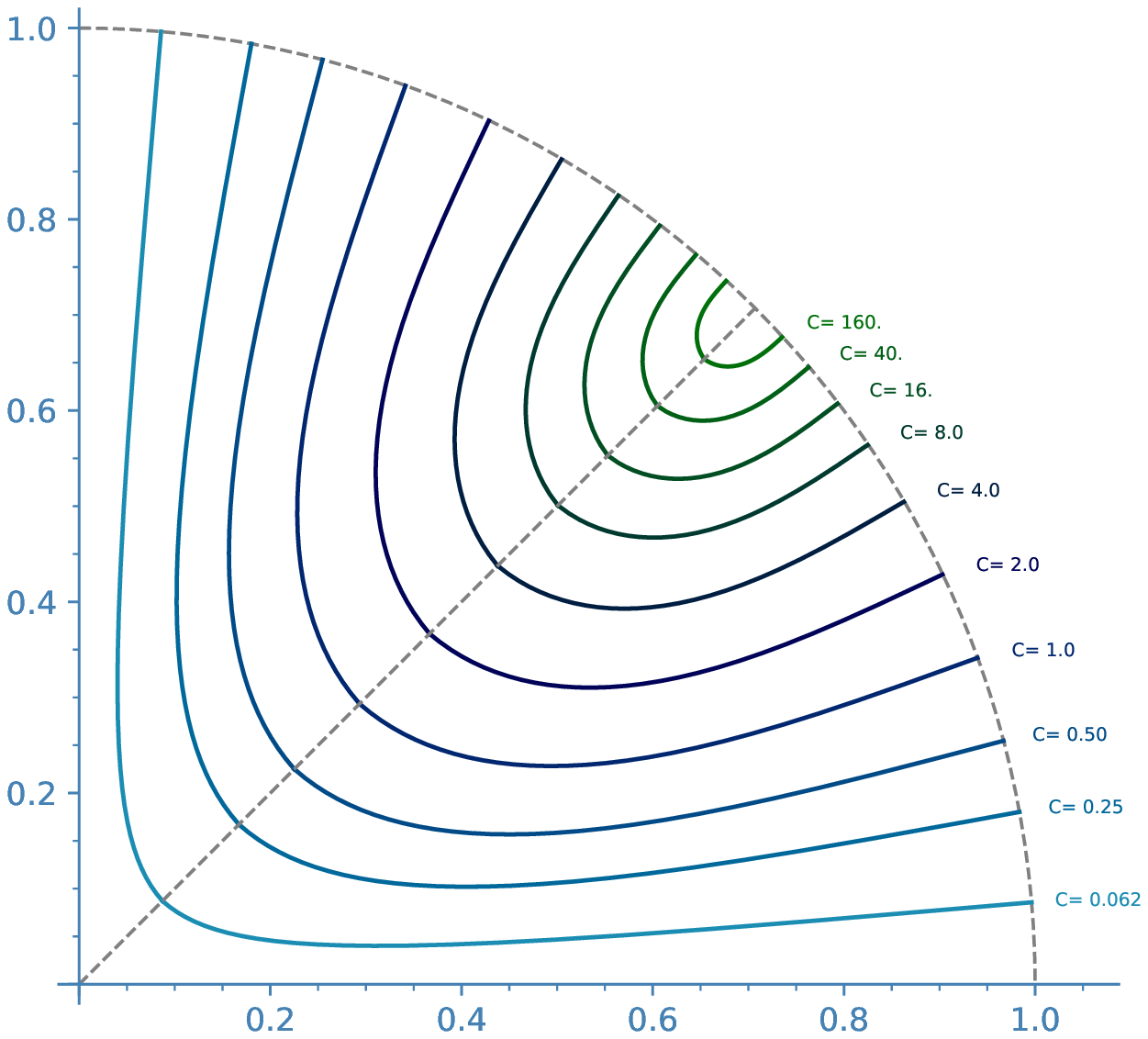} }}%
    \qquad
    \subfloat[in \( {S}^4 \) with \( \theta_C\approx \frac{2\pi}{3} \). The unit circle represents the Clifford torus.]{{\includegraphics[width=6cm]{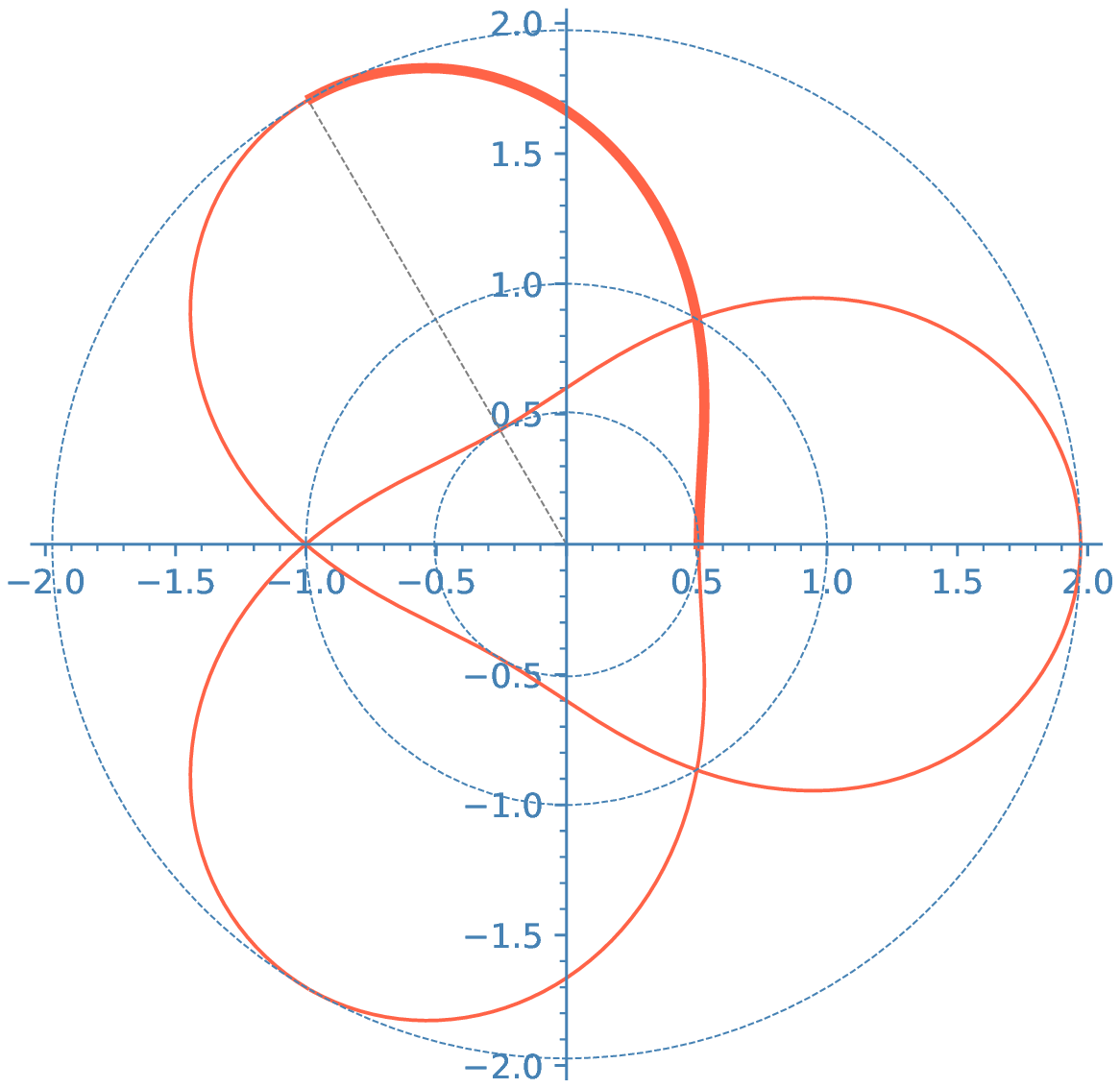} }}%
    \caption{The profile curve of \( M_C \)}%
    \label{fig:min-HS}%
  \end{figure}

  \begin{figure}
    \centering
    \begin{minipage}{0.48\textwidth}
      \centering
        \includegraphics[width=\textwidth]{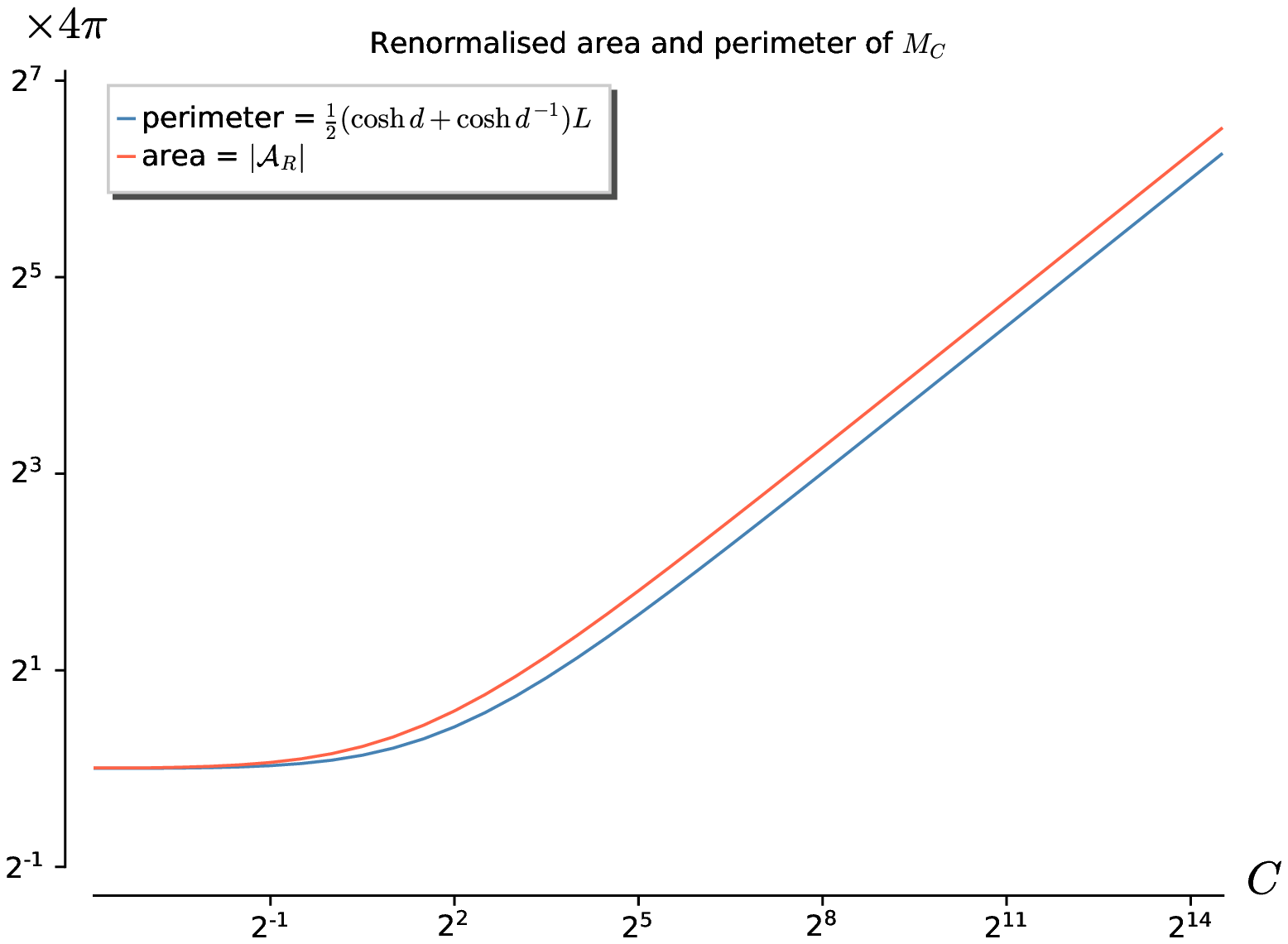} % first figure itself
        \caption{The area and perimeter terms of
          \eqref{eq:upper-AR-a} for the annuli \( M_c \) in \( \mathbb{H}^4 \).}
        \label{fig:2}
    \end{minipage}\hfill
    \begin{minipage}{0.48\textwidth}
        \centering
        \includegraphics[width=\textwidth]{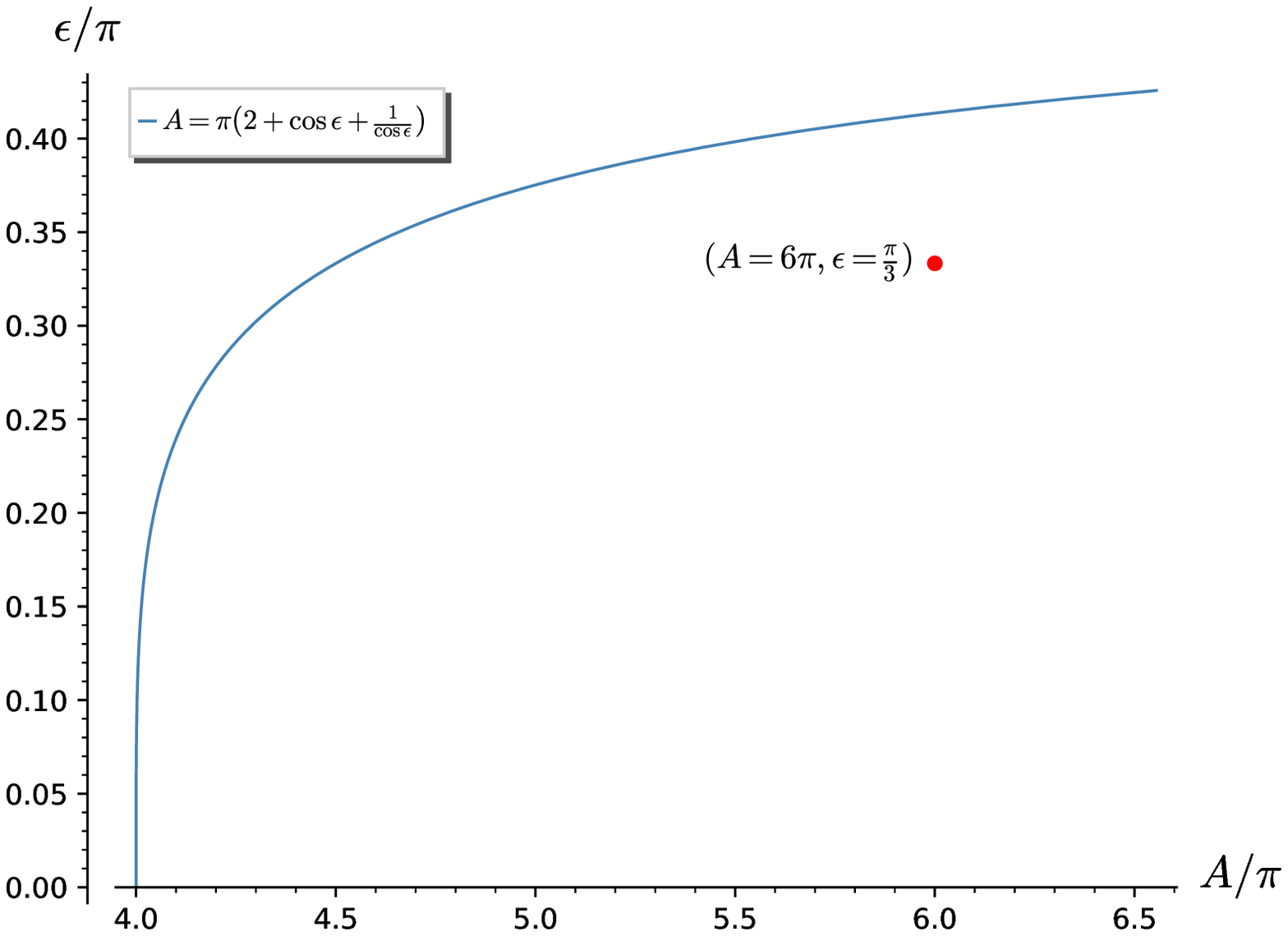} % second figure itself
        \caption{The function \(\epsilon(A,2)\). The red point is the data for Veronese surface.}
        \label{fig:3}
    \end{minipage}
  \end{figure}

  \begin{figure}
    \centering
    \begin{minipage}[b]{0.45\textwidth}
      \centering
    \centering
    \includegraphics[width=4cm]{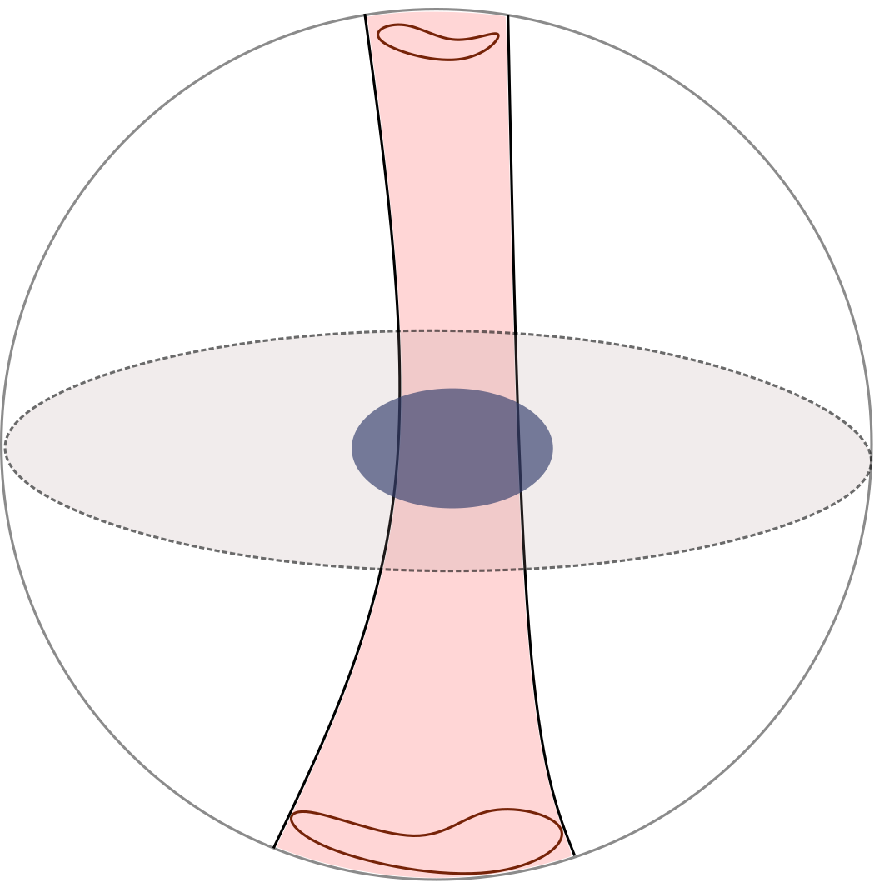} 
    \caption{Any minimal surface \( \Sigma \) filling the link \( L \) (red) is contained in its visual
    hull (pink). It cannot intersect the blue region because of the visual hull. So \(
    \Sigma \) has no intersection with the equatorial disc and has to be disconnected.}
    \label{fig:vishull}
    \end{minipage}\hfill
    \begin{minipage}[b]{0.55\textwidth}
        \centering
    \includegraphics[width=5.5cm]{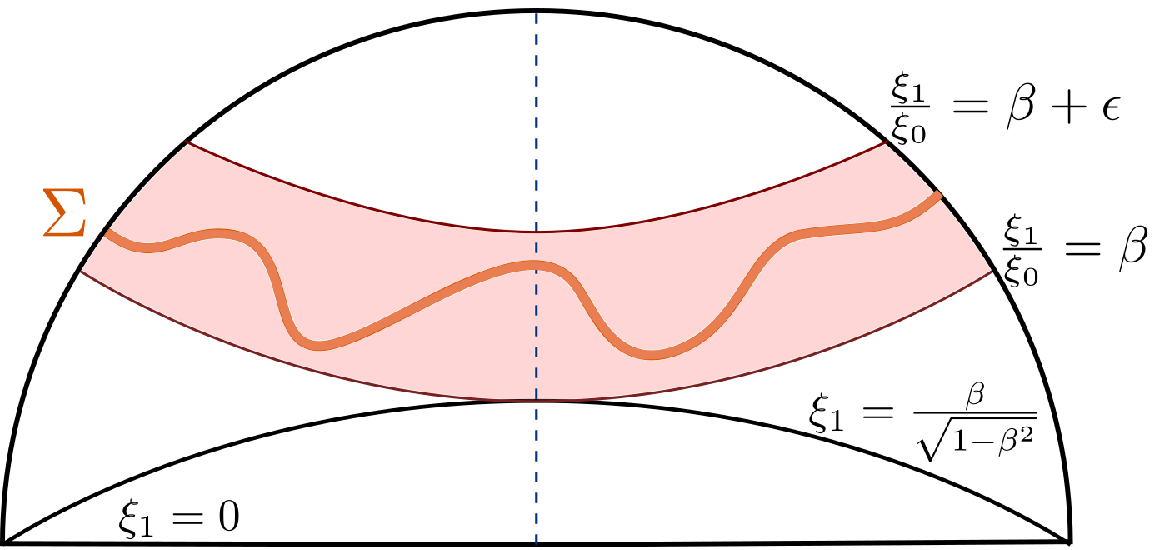} 
    \caption{For a pair \( (\xi_0, \xi_1) \) of time and space coordinates in the same
      Minkowskian system, the function \(\frac{\xi_1}{\xi_0}\) foliates \(\mathbb{H}^n\)
      into totally geodesic leaves. By a Mobius transform, \(\Sigma\) can be put in the
      region \(\beta\leq \frac{\xi_1}{\xi_0}\leq \beta +\epsilon\), which guarantees
      \(\xi_1\geq\ \frac{\beta}{\sqrt{1-\beta^2}}\) on \( \Sigma \). Plugging this into
      \eqref{eq:upper-AR-space} and sending \(\epsilon\) then \( \beta \) to 0, one has
      \(\mathcal{A}_R(\Sigma)\leq -2\pi\).}
    \label{fig:space-est}
    \end{minipage}
  \end{figure}
  
\section{The hyperbolic space and the sphere as warped spaces.}
\label{sec:orga423e7a}
A metric on a Riemannian manifold \(M=N\times [a,b]\) is a \emph{warped product} if it has the
 form 
\begin{equation}
\label{eq:warped-metric}
 g = dr^2 + f^2(r) g_N 
\end{equation}
where \(r\in [a,b]\) and \(g_N\) is a Riemannian metric on \(N\). It can be checked that an anti-derivative \(h\) of the warping function \(f\) satisfies  \(\hess(h) =
f'(r) g\). On the other hand, if such function \(h\) exists, the space can be written
locally as a warped product by the level sets of \( h \).

\begin{proposition}[cf. \cite{Cheeger.Colding96_LowerBoundsRicci}]
\label{prop:cheeger-colding}
Let \(h\) be a \( C^2 \) function on \((M,g)\) with no critical point. Suppose that the level
sets of \(h\) are connected and that
\begin{equation}
\label{eq:xi}
\hess h = U.g
\end{equation}
for a function \(U\in C^0(M)\). Then:
\begin{enumerate}
\item \(U = U(h)\) is a function of \( h \), i.e. a precomposition of \(h\) by a function
  \(U: \mathbb{R}\longrightarrow \mathbb{R}\). So is the function \(V:= |dh|^2\in C^1(M)\)
  and we have \(\frac{dV}{dh} = 2U\).
\item The metrics \(g_a, g_b\) induced from \(g\) on level sets \(h^{-1}(a)\) and \(h^{-1}(b)\) are related
by \(\frac{g_a}{V(a)}= \frac{g_b}{V(b)}\) via the inverse gradient flow of \(h\). 
This defines a metric \(\tilde g\) on level sets of \( h \) under which the flow is isometric. 
The the metric \(g\) pulls back, via the flow map \(h^{-1}(a)\times
{\mathrm{Range}}(h) \longrightarrow M\), to
\begin{equation}
\label{eq:warped-g}
 g = \frac{V(h)}{V(a)}g_a + \frac{dh^2}{V(h)} = V(h)\tilde g + \frac{dh^2}{V(h)}
\end{equation}
which is a warped product after a change of variable \(dr = \frac{dh}{V(h)^{1/2}}\).
\end{enumerate}
\end{proposition}

\begin{proof}
For any vector field \(v\), one has 
\begin{equation}
\label{eq:vv}
 v(V) = 2 g(\nabla_v \nabla h, \nabla h)= 2 \hess(h)(v, \nabla h) = 2U g(v, \nabla h)
\end{equation}
It follows, by choosing \(v\) in \eqref{eq:vv} to be any vector field tangent to level
sets of \(h\), that \(V\) is constant on the level sets of \( h \). Then choose \( v \) to be the inverse
gradient \(\frac{\nabla h}{|\nabla h|^2}\), we have \((\nabla h) V = 2U |\nabla h|^2\), so \( U \) is
also a function of \( h \) and \( \frac{dV}{dh} = 2U \).

For the second part, let \(v_t\) be a vector field of \(M\) tangent to level sets \(h^{-1}(t)\) given by 
pushing forward via the flow of \(u\) a vector field \(v_a\) tangent to \(h^{-1}(a)\). By
definition of Lie bracket, \([v_t, u] = 0\) for all \(t\), and so
\[
 \frac{d}{dt}|v_t|^2 = 2 g(\nabla_u v_t, v_t) = 2 g(\nabla_{v_t}u, v_t) =
\frac{2}{|\nabla h|^2}\hess(h)(v_t, v_t) = \frac{V'}{V}|v_t|^2.
\]
Here in the third equality, we used the fact that \( v_t \) is orthogonal to the gradient of \( h \).
We conclude that \(\frac{|v_t|^2}{V(t)}\) is constant along the flow and so \(\frac{g_a}{V(a)} = \frac{g_b}{V(b)}\) for all \(a,b\) in the range of \(h\).
\end{proof}

The metric \(\tilde g\) is the \(g_N\) of \eqref{eq:warped-metric} and the restrictions of
\( \tilde g \) and \( g \) on the level sets of \( h \) are conformal. We will also use
\(\tilde g\) to denote the metric \(V(h)^{-1}g\) on \(M\) and will call it the
\emph{normalised metric}. 

In applications, we will only assume that the function \(h\) satisfies
\eqref{eq:xi} on \(M\) and it is allowed to have critical points, as in the following examples.

\begin{exampl}[]
\label{ex:xi-R}
In the Euclidean space, the only functions satisfying \eqref{eq:xi} are the coordinates
\(x_i, i=1,\dots,n\) (and their linear combinations) and the square of distance
\(\rho := \frac{1}{2}\sum_{i=1}^n x_i^2\).
\end{exampl}

\begin{exampl}[]
\label{ex:xi-S}
In the unit sphere \({S}^n = \{(x_1,\dots, x_{n+1})\in
\mathbb{R}^{n+1}:\sum_{i=1}^{n+1} x_i^2 = 1\}\), the \(x_i\)
 satisfy \eqref{eq:xi} with \(U = -x_i, V = 1-x_i^2\). It is more intuitive to choose
 \( h \) to be \( 1-x_i \), for which \( U= 1-h \) and \( V = 2h-h^2 \).
\end{exampl}

\begin{exampl}[]
\label{ex:xi-H}
In the hyperbolic space \(\mathbb{H}^n=\{(\xi_0,\dots,\xi_n)\in \mathbb{R}^{n,1}:
   \xi_0^2 - \sum_{i=1}^n\xi_i^2=1, \xi_0 > 0\}\), the Minkowskian coordinates \(\xi_\alpha\) satisfies \eqref{eq:xi} with \(U=\xi_\alpha, V = \xi_\alpha^2 + |d\xi_\alpha|^2\). Here 
   \(|d\xi_\alpha|^2 = -1\) if \( \alpha=0 \) and +1 if \( \alpha \geq 1 \). All null (light-type)
   coordinates can be written, up to a Möbius transform as \( \xi_l = \xi_0 - \xi_1 \). They
   satisfy \eqref{eq:xi} with \( U = \xi_l \) and \( V = \xi_l^2 \).
\end{exampl}

Unlike \( S^n \), \( \mathbb{H}^n \) can be written as warped product in 3
\( SO(n,1) \)-distinct ways. Geometrically, each time coordinate is uniquely defined by an
interior point of \( \mathbb{H}^n \) where it is minimised.  Each space coordinate is
uniquely defined by a totally geodesic hyperplane where it vanishes together with a
coorientation. Each point on \( S_\infty \) defines a null coordinate uniquely up to a
multiplicative factor. In the the half space model with height coordinate \( x>0 \), such
functions are given by \( \frac{\lambda}{x}, \lambda >0 \).

The normalised metrics \(\tilde g\) of \((\mathbb{R}^n, \rho), ({S}^n, x_i)\) and
\((\mathbb{H}^n, \xi_0)\) are all round metrics on \({S}^{n-1}\). For a null coordinate
associated to a point \( b \in S_\infty\), \(\tilde g\) is the flat metric on \(
S_\infty\setminus \{ b\} \). For the space coordinate associated to a hyperplane \( H \),
\( \tilde g \) is obtained by putting hyperbolic metrics on each component of \( S_\infty\setminus \overline{H} \). We
will call this a \textit{doubled hyperbolic} metric. In all cases, \( \tilde g \)
lies in the conformal class at infinity.

\begin{exampl}[]
  \label{ex:xi-fuchsian}
Equation \eqref{eq:xi} is local and so will descend to the quotient whenever there is
an isometric group action that preserves \( h \).
The subgroup of \( SO_+(n,1) \) that preserves a space coordinate \( \xi \) of \( \mathbb{H}^n \) is \( SO_+(n-1,1)
\). When \( n=3 \), the quotient of \( \mathbb{H}^3 \) by a discrete subgroup of \(
SO_+(2,1) \) is called a Fuchsian manifold and the metric descends to a warped product:
\[
g = dr^2 + \cosh^2 r. g_\Sigma\quad\text{on } \Sigma\times \mathbb{R}, \qquad \xi = \sinh r
\]
Here \( \Sigma\) is a Riemann surface of genus at least \( 2 \), equipped
with the hyperbolic metric \( g_\Sigma \).
\end{exampl}

\section{Monotonicity Theorems and Comparison Lemma}
\label{sec:orga12f250}
Given a function \( h \) on a Riemannian manifold \( M \) and a submanifold \( \Sigma \), we write \( \int_{\Sigma,
h\leq t} \) and \( \int_{\Sigma, h=t} \) for the integration over the sub-level
\( h^{-1}(-\infty,t] \) and the level set \( h=t
\) of \( h \) on \( \Sigma \).
The gradient vector of \( h \) in \( M \) is denoted by \(
\nabla h \) and its projection to the tangent of \( \Sigma \) by \( \nabla^\Sigma h
\). We will write \( \Delta_\Sigma h = \dive_\Sigma \nabla^\Sigma h \) for the rough Laplacian of \( h \) on \( \Sigma \).
The volume of the \( k \)-dimensional unit sphere will be denoted by \( \omega_k \).
We will fix a function \( h \) on \( M \) that satisfies property \eqref{eq:xi}.

\begin{definition}[]
  \label{def:tube}
  Let \( \gamma \) be a \( (k-1)\)-submanifold of \( M \). The \emph{forward \( h \)-tube} \( T_\gamma^+ \)
(respectively \emph{backward \( h \)-tube} \( T_\gamma^- \))  built upon \( \gamma \) is the \( k
\)-submanifold obtained by flowing \( \gamma \) along the gradient
of \( h \) forwards (respectively backwards) in time. We will denote their union by \(
T_\gamma \) and the intersection of \( T_\gamma \) with the region \( t_1\leq h \leq t_2
\) by \( T_\gamma(t_1,t_2) \).

The \emph{tube extension} of a
submanifold is its union with the forward tube built upon its boundary, as shown in Figure
\ref{fig:tube-extension}.
\end{definition}

\begin{figure}[hbt!]
    \centering
    \includegraphics[width=5.5cm]{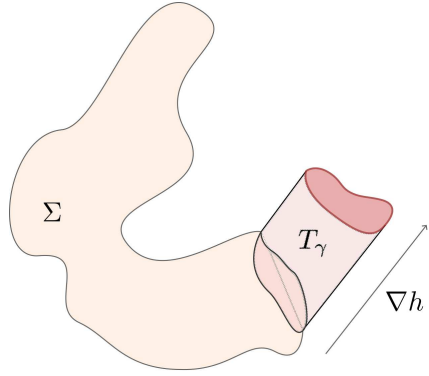} 
    \caption{Extension of a submanifold by gradient tube of \( h \).}
    \label{fig:tube-extension}
\end{figure}

Let \( \gamma \) be a \( (k-1) \)-submanifold of a non-critical level set of \( h=\),
and \(\mathcal{D}\) be a distribution of \(k\)-dimensional planes of \(M\) along \( \gamma
\).

\begin{definition}[]
The \emph{\(\mathcal{D}\)-parallel volume} of \(\gamma\) is
   \[
     \left|\gamma^{\mathcal{D}}  \right|:= \int_{\gamma}\cos \angle(\nabla h, \mathcal{D}) \vol^{k-1}_{\tilde g},
   \]
where the integral was taken with the volume form of the normalised metric \(\tilde g\)
and \(\angle(\nabla h, \mathcal{D})\) is the angle between \(\nabla h\) and \(\mathcal{D}\).
When \(\nabla h\) is contained in \(\mathcal{D}\) at every point of \(\gamma\), \( |\gamma^{\mathcal{D}}| \)
is the \(\tilde g\)-volume of \( \gamma \).
\end{definition}

We will suppose that \( \Sigma \) is a \( k \)-submanifold in the region \( h\geq h_0 \)
of \( M \) and that the intersection of \( \Sigma \) and the level set \( h=h_0 \) has
zero \( k \)-volume. This happens for example when the intersection is empty or a
transversally cut out \( (k-1) \)-submanifold \( \gamma_0 \).

\begin{definition}[]
  \label{def:weight}
  A \emph{weight} is a pair \( (P,c) \) of a function
  \( P: \mathbb{R} \longrightarrow \mathbb{R} \) and a real number \( c\). We define the
  \emph{\( (h_0,(P,c)) \)-volume} of \( \Sigma \) by
\begin{equation}
\label{eq:p-c-weighted}
A_{h_0,(P,c)}(\Sigma)(t) := \int_{\Sigma, h_0\leq h\leq t} P(h) +\frac{c}{k}
|\gamma_0^{T\Sigma}|\ V(h_0)^{k/2}
\end{equation}
\end{definition}

Most of the time, we fix the starting level \( h_0 \) and change the weight and so will often drop
\( h_0 \) in the notation. When \( h_0 \) is chosen to be a
critical value of \( h \), the RHS of \eqref{eq:p-c-weighted} no
longer depends on \( c \) and we drop it in the data of a weight.

Because the metric \( g \) has the form \eqref{eq:warped-g}, the \( (P,c) \)-volume of a
tube \( T_\gamma(h_0,t) \), as a function of \( t \), is independent of the choice of \(
\gamma \) up to a
multiplicative factor:
\begin{equation}
\label{eq:weight-tube}
A_{(P,c)}(T_{\gamma}(h_0,t)) = \frac{|\gamma|_{\tilde g}}{\omega_{k-1}}
Q_{(P,c)}(t),\qquad Q_{(P,c)}(t):= \omega_{k-1}\left[ \int_{h_0}^t P(h)V^{\frac{k}{2} - 1}(h) dh + \frac{c}{k}V(h_0)^{\frac{k}{2}}\right]
\end{equation}

\begin{definition}[]
  \label{def:density}
The \emph{\( (P,c) \)-density} of \( \Sigma \) is obtained by normalising its weighted
volume by that of a tube:
\begin{equation}
\label{eq:weight-density}
\Theta_{(P,c)}(t):= \frac{A_{(P,c)}(\Sigma)(t)}{Q_{(P,c)}(t)}.
\end{equation}
\end{definition}

Clearly, the density of a tube is constant in \( t \).  In \eqref{eq:weight-tube}, we
choose to put the constant \( \omega_{k-1} \) in \( Q \) to be consistent with the
classical monotonicity theorem (\( h = \rho \) in Example \ref{ex:xi-R}). In the classical
case with \( h_0=0 \), \( h \)-tubes are visually cones or discs.

We call the pair \( (U, 1) \) the  \emph{natural weight}. The naturally weighted
volume of tubes simplifies to \( Q_{(U,1)}(t) = \frac{\omega_{k-1}}{k}|\gamma|_{\tilde
  g}V^{k/2}(t)  \), which is independent of the starting level \( h_0 \).

\begin{theorem}[Naturally weighted monotonicity]
\label{thm:monotonicity-warped}
Let \(h\) be a \(C^2\) function on a Riemannian manifold \(M\) that
satisfies \eqref{eq:xi} with \(U\) and \(V= |\nabla h|^2\) being functions of \(h\)
such that \(U = \frac{1}{2}V'\). Let \(\Sigma\) be a minimal
\( k \)-submanifold in \(M\). Then the naturally weighted density
\(\Theta_{(U,1)}(\Sigma)(t)\) is increasing in region \( U>0 \) and decreasing in region \( U<0 \).

Moreover, the conclusion also holds for the tube extension
of minimal submanifolds. 
\end{theorem}
\begin{proof}
The Laplacian of \( h \) on \( \Sigma \) can be computed to be
  \[
    \Delta_\Sigma h = kU.
  \]
This can be seen via the formula:
\begin{lemma}[Leibniz rule]
\label{lem:chain-rule}
Let \(f:(\Sigma,g_\Sigma) \longrightarrow (M, g)\) be a map between Riemannian manifolds and \(\tau(f)\) be its tension
field. Then for any \(C^2\) function \(h\) on \(M\):
\begin{equation}
\label{eq:chain-rule}
\Delta_{\Sigma} (h\circ f) = \tr_{\Sigma} f^*\hess h + dh.\tau(f)
\end{equation}
In particular,
\[
  \Delta_\Sigma h = \tr_\Sigma \hess h
\]
if \(\Sigma\) is either minimal or tangent to the gradient of \(h\).
\end{lemma}

Because the volume forms of \(g\) and \(\tilde g\) on \(\gamma\) are related
by \(\vol_{g}^{k-1} = V^{\frac{k-1}{2}}\vol_{\tilde g}^{k-1}\) and 
\(\cos \angle(\nabla h, T\Sigma).V^{1/2} = |\nabla^\Sigma h|\), the second term in the
definition of weighted volume \eqref{eq:p-c-weighted} is \( \frac{c}{k}\int_{\gamma_0}|\nabla^\Sigma h|  \).
By Stokes' theorem,
\begin{equation}
\label{eq:stokes}
kA_{(U,1)}(\Sigma)(t) = \int_{\Sigma,h_0\leq h\leq t} \Delta_\Sigma h + \int_{\gamma_0}|\nabla^\Sigma h| = \int_{\Sigma,h=t}(\nabla^\Sigma h)\cdot n = \int_{\Sigma,h=t}|\nabla^\Sigma h|.
\end{equation}
The last equality is because the outer normal of the sublevel set \(h\leq t\) in \(\Sigma\) is \(n=
\frac{\nabla^{\Sigma} h}{|\nabla^{\Sigma} h|}\). By the coarea formula, 
\[
 \frac{dA_{(U,1)}(t)}{dt} = U(t)\int_{\Sigma,h = t} \frac{1}{|\nabla^\Sigma h|}.
\]
Combining this with \eqref{eq:stokes} and \(|\nabla^\Sigma h|^2\leq  V(h)\), one has \(\frac{1}{U}\frac{dA_h}{dt} \geq \frac{k A_h}{V}\), or
\[
\frac{1}{U} \frac{d}{dt}\left( \Theta_{(U,1)}(\Sigma)\right)\geq 0.
\]

For the extension \( \tilde\Sigma \) of \( \Sigma \) by a forward tube built upon \( \gamma = \partial\Sigma \), it suffices to rewrite \eqref{eq:stokes} as
\begin{align*}
 kA_{(U,1)}(\tilde \Sigma)(t) &= \left(\int_{\Sigma, h\leq t}+
           \int_{T^+_{\gamma}, h\leq t}\right)\Delta h \\
  &=\int_{\Sigma,h=t}|\nabla^{\Sigma} h| +  \int_{\gamma\cap \{h\leq t\}}|\nabla^{\Sigma} h|+ 
   \int_{T^+_\gamma, h=t} |\nabla h| - \int_{\gamma\cap \{h\leq t\}} |\nabla h| \\
  &\leq  \int_{\Sigma,h=t}|\nabla^{\Sigma} h| +   
   \int_{T_\gamma, h=t} |\nabla h|  = 
   \int_{\tilde\Sigma,h=t}|\nabla^{\tilde\Sigma} h|.
\end{align*}
\end{proof}

\begin{remark}
\label{rem:monotonicity-warped}
We only need the "\(\leq\)" sign in \eqref{eq:stokes} and hence it suffices that \(\Delta_\Sigma h \geq k U(h)\). Theorem \ref{thm:monotonicity-warped}
also holds if \(\hess h\geq
   U.g\), provided that \(U\) and \(V = |d h|^2\) are 
functions of \(h\) with \(U = \frac{1}{2}V'\) (see Section \ref{sec:bounded-curv}).
\end{remark}

The classical monotonicity theorem is recovered by applying Theorem
\ref{thm:monotonicity-warped} to the function \( \rho \) of Example \ref{ex:xi-R}. On the other hand, Choe and
Gulliver \cite[Theorem 3]{Choe.Gulliver92_IsoperimetricInequalitiesMinimal} discovered
that there are monotonicity theorems in \( S^n \) and \( \mathbb{H}^n \) if one weights
the volume functional by the cosine (respectively hyperbolic cosine) of the distance
function. These theorems come from applying Theorem \ref{thm:monotonicity-warped} to
Euclidean coordinates of Example \ref{ex:xi-S} and time
coordinates of Example \ref{ex:xi-H}.

Since the proof of Theorem \ref{thm:monotonicity-warped} only uses integration by part, it
can be adapted to stationary rectifiable \( k \)-currents or \( k \)-varifolds. Instead of
integrating the Laplacian of \( h \), it suffices to apply the first variation formula to
a test vector field that is a suitable cut-off of the gradient of \( h \). See
\cite[Theorem 1]{Anderson82_CompleteMinimalVarieties} and
\cite[\(\S\)7]{Ekholm.etal02_EmbeddednessMinimalSurfaces}

Theorem \ref{thm:monotonicity-warped} can also be adapted for harmonic maps. Given a map \(f:\Sigma \longrightarrow M\), we
define its \emph{dimension at a point} \(p\in \Sigma\) to be the ratio \(\frac{|df_p|^2}{|df_p|_o^2}\)
of the tensor norm of the derivative at \(p\) (also called \emph{energy
density}) and its operator norm, or \(+\infty\) if the latter vanishes. Note that when \(df_p\) is conformal, 
this is the dimension of \(\Sigma\). The \emph{dimension} of \(f\), defined as the smallest
dimension at all points, plays the role of \(k\) in the argument.

The \emph{naturally weighted Dirichlet energy} and \emph{naturally weighted density}  of \(f\)
are defined as
\[
  E_{(U,1)}(t):=
\int_{\Sigma, h_0\leq h\circ f\leq t} U|df|^2 + \int_{\Sigma,h\circ f = h_0}|d(h\circ
f)|,\quad \Theta_{(U,1)}(t):= \frac{E_{(U,1)}(t)}{V(t)^{k/2}}.
\]

\begin{theorem}
\label{thm:monotonicity-map}
Let \(h, U, V\) be as in Theorem \ref{thm:monotonicity-warped} and
\(f:\Sigma \longrightarrow M\) be a harmonic map. Then \(\frac{d}{dt}\Theta_h\) has the same sign as \(U\). 
\end{theorem}
\begin{proof}
By Lemma \ref{lem:chain-rule}, \(\Delta (h\circ f) = U |df|^2\) and by
integration by part, \(E_{(U,1)}(t) = \int_{\Sigma, h\circ f = t} |d(h\circ f)|\). One then
compares \(E_{(U,1)}\) with its derivative obtained from the coarea formula \(\frac{d}{dt}E_{(U,1)} = U(t) \int_{\Sigma, h\circ f = t}
\frac{|df|^2}{|d(h\circ f)|}\).
The definition of 
\(k\) guarantees \(\frac{|df|^2}{|d(h\circ f)|} \geq k \frac{|d(h\circ f)|}{|dh|^2}\) and
therefore
\[
  U^{-1}\frac{d}{dt}E_{(U,1)}(t)\geq \frac{k}{V}E_{(U,1)},
\]
which is equivalent to \( \frac{1}{U}\frac{d}{dt}\Theta_{h} \geq 0 \).
\end{proof}

We will give another characterisation of the monotonicity in the extended region.  For
this, we suppose that \( \Sigma \) is contained in the region \( h_0 \leq h\leq h_1 \) as
in Figure \ref{fig:tube-2},
with boundary composing of 2 parts: one (possibly empty) in level \( h=h_0 \) and one,
called \( \gamma \), in \( h=h_1 \). Let \( \tilde\Sigma \) be the forward tube extension
of \( \Sigma \) built upon \( \gamma \). Clearly the monotonicity of \( \tilde \Sigma \)
and \( \Sigma \), under any weight, are the same in the region \( [h_0, h_1] \). In the
region \( h\geq h_1 \), the monotonicity of \( \tilde\Sigma \) is equivalent to a volume
comparison between \( \Sigma \) and the tube.
\begin{figure}[hbt!]
    \centering
    \includegraphics[width=6cm]{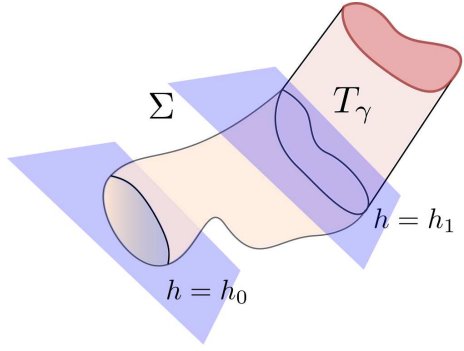} 
    \caption{Illustration of Lemma \ref{lem:less-than-tube}.}
    \label{fig:tube-2}
\end{figure}

\begin{lemma}
  \label{lem:less-than-tube}
Let \(t_1 < t_2\) be two numbers in \([h_1,+\infty)\), then the following statements are
equivalent:
\begin{enumerate}
\item \(\Theta_{(P,c)}(\tilde\Sigma)(t_1) \leq \Theta_{(P,c)}(\tilde\Sigma)(t_2)\) 
\item \(A_{(P,c)}(\Sigma)(h_1) \leq A_{(P,c)}(T_\gamma(h_0, h_1))\)
\end{enumerate}
\end{lemma}
\begin{proof}
It follows from straightforward volume addition and the fact that
the density of a tube is constant.
\end{proof}

By Theorem \ref{thm:monotonicity-warped}, we are mostly interested in the region of
\( M \) where \( U \) is signed. By switching the sign of \( h \), we can assume that
\( U\geq 0 \). When one applies Lemma \ref{lem:less-than-tube} to minimal submanifold, the assumption that \( \Sigma \) belongs to the
region \( h_0\leq h\leq h_1 \) becomes automatic. This is because
\( h\leq h_1 \) on \( \partial\Sigma \) and its restriction to \( \Sigma \) is
superharmonic \( \Delta h = kU \geq 0 \). Also, there is no
critical point of \( h \) in the region \( U > 0 \) because the function \( V(h) = |dh|^2 \) satisfies
\( V' = 2U>0 \).

Let \( (P_1,c_1), (P_2,c_2) \) be 2 positive weights, that is \( P_i\geq  0 \). The tube volumes, defined in
\eqref{eq:weight-tube}, will be denoted by \(Q_1(t), Q_2(t)\).

\begin{definition}
We say that \((P_1, c_1)\) is \emph{weaker} than \((P_2,c_2)\) and write \( (P_1,c_1) \ll
(P_2,c_2) \) if \(\frac{P_1}{Q_1}\leq
\frac{P_2}{Q_2}\) for all \( t \), Equivalently, this means if \(\frac{d}{dt}\frac{Q_1}{Q_2}\leq 0\), i.e. the
\((P_2,c_2)\)-volume of a \(k\)-dimensional tube increases faster than its \((P_1,c_1)\)-volume.
\end{definition}

\begin{lemma}[Comparison]
\label{lem:comparison}
Let \(\Sigma\subset M^n\) be a \( k \)-submanifold not necessarily minimal and
\((P_1,c_1), (P_2,c_2)\) be two positive weights. Let \(\Theta_1,\Theta_2\) be the corresponding densities. 
\begin{enumerate}
\item Suppose that \((P_1,c_1)\ll (P_2,c_2)\) and that \(\Theta_2(\Sigma)\) is
  increasing. Then \( \Theta_1 \) is also increasing and we have \(\Theta_1\leq\Theta_2\).
\item On the other hand, if \((P_2, c_2) \ll (P_1, c_1)\) and \(\Theta_2(\Sigma)\) is
  increasing. Then we still have \(\Theta_1\geq\Theta_2\).
\end{enumerate}
\end{lemma}
\begin{proof}
One has \(P_1^{-1} \frac{dA_{(P_1,c_1)}}{dt} = P_2^{-1}\frac{dA_{(P_2,c_2)}}{dt}\) from coarea
formula. Therefore
\begin{equation}
\label{eq:comparison-1}
\frac{Q_1}{P_1}\frac{d\Theta_1}{dt} + \omega_{k-1}V^{\frac{k}{2}-1}\Theta_1 = \frac{Q_2}{P_2}\frac{d\Theta_2}{dt} + \omega_{k-1}V^{\frac{k}{2}-1}\Theta_2
\end{equation}
which can be rearranged into
\begin{equation}
\label{eq:comparison-2}
P_1^{-1} \frac{d}{dt}\left( Q_1(\Theta_1 -\Theta_2) \right) = \left(\frac{Q_2}{P_2}-\frac{Q_1}{P_1}\right) \frac{d\Theta_2}{dt}
\end{equation}
For the second part of the Lemma, it follows from the hypothesis  that the RHS of
\eqref{eq:comparison-2} is positive, and therefore \(Q_1(
\Theta_1- \Theta_2)\) is an increasing function. The latter vanishes at \(t=0\), which implies \(\Theta_1\geq\Theta_2\)
for all time. For the first part, the RHS of \eqref{eq:comparison-2} is negative by hypothesis and
 therefore \(\Theta_2 \geq \Theta_1\). The rest of the conclusion follows by substituting
 this into \eqref{eq:comparison-1}.
\end{proof}

The previous proof can also be adapted for rectifiable varifolds and currents. It suffices
to interpret equation \eqref{eq:comparison-1} in weak sense and choose a suitable test function.

It follows from Theorem \ref{thm:monotonicity-warped}, Lemma \ref{lem:comparison} and Lemma
\ref{lem:less-than-tube} that:
\begin{corollary}[]
Let \( \Sigma \) be a minimal \( k \)-submanifold whose boundary is a union of its
intersection with \( h = h_0\)) and a \( (k-1) \)-submanifold \( \gamma \) of the level set \( h=h_1
\). Then for any weight \( (P,c) \) weaker than \( (U, 1) \), we have
\[
A_{(P,c)}(\Sigma)\leq A_{(P,c)}(T_\gamma(h_0, h_1)).
\]
\end{corollary}

We will now give a few examples to which Lemma \ref{lem:comparison} applies.
First, for any positive \( P \), \( (P, c_1) \ll (P, c_2) \) if and only if \( c_1 > c_2
\).
A less obvious example is in the unit ball of \(\mathbb{R}^n\) equipped with the Poincaré metric \(g_{\mathbb{H}}=
   \frac{4}{(1-r^2)^2}g_E\). We choose \( h = \xi \) the time coordinate minimised at the
   origin \( O \), given by \(\xi := \frac{1+r^2}{1-r^2}\) with starting level \( h_0=1
   \). Since \( h_0 \) is a critical value of \( \xi \), we can forget the constant \( c
   \) in the data of a weight. We will define and compare 5
   different weights \( P_i,i=1,\dots,5 \), starting with the natural weight \( P_1 =\xi
   \) and the uniform weight \( P_2=1 \).

   The Euclidean \( k \)-volume corresponds to
   weight \(P_3 = (1+\xi)^{-k}\).
   The sphere \( S^n \) is the 1-point compactification of \(
   \mathbb{R}^n \) and the round metric \( g_S \) induces a \( k \)-volume functional, which
   under \( g_{\mathbb{H}} \) corresponds to a weight \(P_4=\xi^{-k}\).
   On this sphere, let \( x \) be the Euclidean coordinate (see Example
   \ref{ex:xi-S}) that is maximised at \( O \). The weighted \(g_S\)-volume corresponds to \(P_5 = \xi^{-k-1}\).
   It can be checked that \(P_{i+1}\) is weaker than \(P_i\).

   By Lemma \ref{lem:comparison}, we have the following chain of monotonicity:
\begin{equation}
\label{eq:chain}
\text{time-weighted }  g_{\mathbb{H}^n} \gg   \text{unweighted } g_{\mathbb{H}^n}
\gg g_E  \gg \text{unweighted } g_{{S}^n}   \gg \text{weighted } g_{{S}^n} 
\end{equation}
where any submanifold \(\Sigma\subset \mathbb{B}^n\) having increasing density with one volume
functional in the chain will automatically have increasing density in any other volume functional
following it and that the densities can be compared accordingly. More generally:

\begin{lemma}[]
\label{lem:compensated-comparison}
\begin{enumerate}
\item Let \( M=\mathbb{H}^n \), \( h \) be a Minkowskian coordinate. Then natural weight is stronger
  than the weight  \((1, h_0^{-1}) \),
\item Let \( M =  S^n \) and \( h = 1-x \) where \( x \) is a Euclidean coordinate. Then natural weight is weaker than
  the weight \( (1, (1-h_0)^{-1}) \).  
\end{enumerate}
\end{lemma}
\begin{proof}
Both statements are inequalities on explicit functions of 1 real variable: denote the
weights by \( (P_1,c_1),(P_2,c_2) \) and tube volumes by \( Q_1,Q_2 \), we want to check
that \( \frac{Q_1}{P_1}\leq \frac{Q_2}{P_2} \). In both cases, it suffices to prove that \(
\frac{d}{dt}\frac{Q_1}{P_1}\leq \frac{d}{dt}\frac{Q_2}{P_2} \) and that \(
\frac{Q_1}{P_1}= \frac{Q_2}{P_2} \) at \( h_0 \). 
\end{proof}

Theorems \ref{thm:intro-time}, \ref{thm:intro-space}, \ref{thm:intro-null}
in the introduction follow from Theorem \ref{thm:monotonicity-warped}, Lemma \ref{lem:comparison} and Lemma \ref{lem:compensated-comparison}.

\section{Applications to minimal surfaces in \(\mathbb{H}^n\) and \( S^n \)}
\label{sec:orga09937d}\subsection{Volume estimates}

The first application of Theorem \ref{thm:monotonicity-warped} is to estimate the
intersection between a minimal \( k \)-submanifold and level sets of \( h \).

\begin{corollary}
\label{cor:est-curve}
Suppose, in addition to the hypothesis of Theorem \ref{thm:monotonicity-warped}, that \(
U\geq 0 \) in the interval \( [h_0, t] \).
Let \( \gamma_t \), \( \gamma_0 \) be the intersection of \( \Sigma \) and the level sets \( h=t \), \( h=h_0 \) and \(
|\gamma_t|, |\gamma_0^{T\Sigma}| \) be their volume and parallel volume under the metric
\( \tilde g \). Then \(|\gamma_t| \geq |\gamma_0^{T\Sigma}|.\)
\end{corollary}

\begin{proof}
By Lemma \ref{lem:less-than-tube}, the naturally weighted density at \( t \) is less
than \( \omega_{k-1}^{-1}|\gamma_t| \), but it is \(\omega_{k-1}^{-1} |\gamma_0^{T\Sigma}| \) at \( h_0 \).
\end{proof}

When the level set \( h=h_0 \) is a critical and degenerate to a point, as in the case of
time coordinates of \( \mathbb{H}^n \) and Euclidean coordinates of \( S^n \),
Corollary \ref{cor:est-curve} becomes:

\begin{corollary}[]
\label{cor:est-curve-degen}
Let \( p \) be a point on a minimal \( k \)-submanifold \( \Sigma \) of \( \mathbb{H}^n \)
(respectively \( S^n \)) and \( \gamma \) be the intersection of \( \Sigma \) and a
geodesic sphere \( S \) of radius \( r \in (0,+\infty) \) (respectively \(r\in (0, \pi/2]
\)) centred at \( p \). Then the \( (k-1) \)-volume of \( \gamma \) is at least that of a
great \( (k-1) \)-subsphere of \( S \).

In particular, a complete minimal submanifold of \( \mathbb{H}^n \) is contained in the
visual hull of its ideal boundary.
\end{corollary}

Theorem \ref{thm:monotonicity-warped} can also be used to upper bound the volume of
minimal submanifolds of \( \mathbb{H}^n \) by that of the tube competitors. The following
estimates come from combining Lemma \ref{lem:less-than-tube} and Theorems
\ref{thm:intro-time}, \ref{thm:intro-space}, \ref{thm:intro-null}. Because these
monotonicity theorems hold for tube extensions, the boundary \( \gamma \) does not need to
lie in a level set of the Minkowskian coordinate.

\begin{corollary}[]
  \label{cor:est-vol}
Let \( \xi \) be a Minkowskian coordinate, \( \Sigma \) a minimal \( k \)-submanifold of \( \mathbb{H}^n \) with boundary a \(
(k-1) \)-submanifold \( \gamma \) and \( T_\gamma \) its \( \xi \)-tube. Suppose that \(
\Sigma \) lies in a region \( \xi\geq a \) with \( a>0 \). Then the volume of \( \Sigma \) is at most
the \( (a, (1, a^{-1})) \)-volume of \( T_\gamma \), which is  \(|\gamma|_{\tilde g} Q_\delta(a,t)\) (see \eqref{eq:Qdelta}).
\end{corollary}

When \( \Sigma \) is in the region \( \xi\geq a_1 > a_2 \), the volume
estimate obtained by Corollary \ref{cor:est-vol} with \( a = a_1 \) implies the
version with \( a=a_2 \). It is thus optimal to choose \( a =\min_\Sigma \xi \).

\subsection{Renormalised isoperimetric inequalities}
\label{sec:renorm-isop-ineq}

Corollary \ref{cor:est-vol} consists of three different upper bounds for volume of minimal
submanifolds, depending on the type of \( \xi \). The version for time coordinate with
\( a=1 \) was proved by Choe and Gulliver and was essential to their proof of the
isoperimetric inequality for minimal surfaces of \( \mathbb{H}^n \)
\cite{Choe.Gulliver92_SharpIsoperimetricInequality}. Complete surfaces, on the other hand,
necessarily run out to infinity and so have infinite area. By studying the area growth,
Graham and Witten defined finite number, called the renormalised area, that is intrinsic
to the minimal surface. We will briefly review their result and use Corollary
\ref{cor:est-vol} to prove three different renormalised versions of the isoperimetric
inequality.

A \emph{boundary defining function} (bdf) of \(\mathbb{H}^n\) is a non-negative function \(\rho\)
on the compactification \(\overline{\mathbb{H}^n}\) that vanishes exactly on \({S}_{\infty}\) and
exactly to first order. 
Such function is called \emph{special} if \(|d\ln\rho|_{g_{\mathbb{H}^n}} = 1\) on a neighbourhood
of the ideal boundary. It was proved in
\cite{Graham.Witten99_ConformalAnomalySubmanifold} that any complete minimal surface \(\Sigma\) of \( \mathbb{H}^n \) that are
\(C^2\) up to its ideal boundary has to following area expansion under a special 
bdf. Here \( A \) denotes the hyperbolic area and \( \bar g = \rho^2 g \).
\begin{equation}
\label{eq:GW-1}
 A(\Sigma\cap \{\rho \geq \epsilon\}) = \frac{|\gamma|_{\bar g}}{\epsilon} +
\mathcal{A}_R + O(\epsilon).
\end{equation}
The coefficient \(\mathcal{A}_R\) is called the \emph{renormalised
area} of \(\Sigma\) and is independent of the choice of \(\rho\). Clearly, when \( \rho
\) is only \emph{special up to third order}, that
is \( \rho = \bar \rho + o(\bar\rho^2) \) for a special bdf \(
\bar\rho \),  the 2 metrics \( \rho^2 g\) and \(\bar\rho^2g \)
are equal on \( S_\infty \) and the expansion \eqref{eq:GW-1} still holds for \( \rho \).

\begin{lemma}[]
\label{lem:bdfs}
For any Minkowskian coordinate \( \xi \), the bdf \( \rho = |\xi|^{-1} \) is
special up to third order.
\end{lemma}
\begin{proof}
  Clearly when \( \xi \) is of null type, \(\rho = \xi^{-1}\) is the half space coordinate
  and so is special. When \( \xi \) is of time type (respectively space type), let \( l \)
  be the distance function towards the defining interior point (or hyperplane). Clearly
  \( |dl| = 1 \) and so \( \bar \rho = 2\exp(-l) \) is a special bdf. It can be checked
  that \( |\xi| = \cosh l \) (respectively \( \sinh l \)), and so
  \( \rho = \bar\rho + O(\bar\rho^3)\).
\end{proof}

Now we suppose that \( \Sigma \) is a complete minimal surface of \( \mathbb{H}^n \) in the region
\( \xi \geq a > 0 \), the expansion \eqref{eq:GW-1} with \( \rho = |\xi|^{-1} \) can be rewritten as 
\begin{equation}
\label{eq:GW-2}
 A(\Sigma \cap \{\xi < t\})(t) = |\gamma|_{\bar g} t + \mathcal{A}_R + O(t^{-1})
\end{equation}
Here \( \bar g = \xi^{-2}g_{\mathbb{H}} = \tilde g \) is either the round, flat or doubled
hyperbolic metric on \( S_\infty \). 
Corollary \ref{cor:est-vol} gives an upper bound of the LHS of \eqref{eq:GW-2}:
\begin{equation}
\label{eq:est-AR-1}
A(\Sigma\cap\{\xi\leq t\}) \leq |\gamma_t|_{\tilde g}\left( \int_{a}^t dh +
  \frac{1}{2a}V(a) \right) = |\gamma| t - |\gamma| \left( \frac{1}{2}a -
  \frac{\delta}{2a}\right) + O(t^{-1})
\end{equation}
For the equality at the end of \eqref{eq:est-AR-1}, we used the fact that \( |\gamma_t|_{\tilde g} = |\gamma|_{\tilde g} + O(t^{-2}) \), which is because minimal surfaces
meet \( S_\infty \) orthogonally. Now plug \eqref{eq:est-AR-1} into
  \eqref{eq:GW-2}, we have the following estimate of \( \mathcal{A}_R \), which
  specialises to Theorems \ref{thm:upper-AR} and \ref{thm:upper-AR-space} in the introduction:

\begin{theorem}[]
\label{thm:compensated-est-AR}
Let \(\xi\) be a Minkowskian coordinate, \( \tilde g \) be the
corresponding round/ double hyperbolic/ flat metrics on \(S_\infty\) and \(\Sigma\) be a complete minimal surface that is \(C^2\) near its ideal boundary
\(\gamma\). Suppose that \( \Sigma \) lies in the region \( \xi\geq a > 0 \), then
\begin{equation}
\label{eq:compensated-AR}
 \mathcal{A}_R (\Sigma) + \frac{1}{2}|\gamma|_{\tilde g} \left(a - \frac{\delta}{a}\right)
\leq 0.
\end{equation}
where \( \delta = -1,0,+1 \) depending on the type of \( \xi \) as in Example \ref{ex:xi-H}.
\end{theorem}

\subsection{Antipodalness of minimal submanifolds of \( S^n \)}
\label{sec:volume-minim-subm}
In \( S^n \), we will use Theorem \ref{thm:monotonicity-warped} with the function \( h =
1-x \) instead of the Euclidean coordinate \( x \)
of \( \mathbb{R}^{n+1} \). Recall that the naturally weighted
monotonicity theorem holds only in one half sphere (\( h\in [0,1] \)) and the natural weight
is weaker than the uniform weight. The
second half of Lemma \ref{lem:comparison} applies and gives an
estimate of the (unweighted) volume of minimal submanifolds. 

Recall that the notion of \( \epsilon \)-antipodal was given in Definition \ref{def:eps-antipodal}.
It is clear that any subset of \( S^n \) is \( \pi \)-antipodal and any closed subset is \( 0 \)-antipodal if and only if it is symmetric by antipodal
map. Closed minimal submanifolds are
\( \frac{\pi}{2} \)-antipodal because they cannot be contained in one half
sphere. This is because for any Euclidean coordinate \( x \),
\begin{equation}
\label{eq:x-weighted-0}
\int_\Sigma x = -\frac{1}{k}\int_\Sigma \Delta x = 0,
\end{equation}
and so \( \Sigma \) cannot be contained in a half sphere where \( x \) is signed. We prove
that the further a minimal submanifold is from being antipodal, i.e. the greater \(
\epsilon \) is, the greater its volume has
to be.

We define \( \epsilon(A,k) \) to be the unique
solution in \( [0,\frac{\pi}{2}) \) of 
\begin{equation}
\label{eq:eps-sphere}
\int_{\epsilon}^{\pi/2}\sin^{k-1}t dt + \frac{(1-\cos^2\epsilon)^{k/2}}{k\cos\epsilon}= \omega_{k-1}^{-1}\left(A - \frac{\omega_k}{2}\right)
\end{equation}
It can be checked that \( \epsilon(\cdot, k) \) is
strictly increasing function on the interval \( [\omega_k,+\infty) \) with \( \epsilon(\omega_k, k) = 0 \) and converges to \( \frac{\pi}{2} \) when \( A\to+\infty
\). When \( k=2 \), \eqref{eq:eps-sphere} simplifies to
\[
\cos\epsilon + \frac{1}{\cos\epsilon} = \frac{A}{\pi}- 2
\]

\begin{proposition}
 \label{prop:eps-sphere}
 A closed minimal \( k \)-submanifold \( \Sigma \) of \( S^n \) with volume \( A \) must be
 \( \epsilon(A,k) \)-antipodal. More generally, if \( \Sigma \) contains a point \( p \)
 with density \( m \), then it cannot avoid the ball of radius
 \( \epsilon(\frac{A}{m},k) \) centred at the antipodal point \( p' \) of \( p \).
\end{proposition}

Before proving Proposition \ref{prop:eps-sphere}, we note that 
\eqref{eq:x-weighted-0} can be interpreted as the equal distribution of weighted volume of
\( \Sigma \) between opposing half spheres. More precisely, if \( h_1=
1-x \) and \( h_2 = 1+x \) are warping functions associated to the North pole \( p \) and South pole \( p' \), then
\[
\int_{\Sigma, 0\leq h_1\leq 1} U(h_1) - \int_{\Sigma, 0\leq h_2\leq 1} U(h_2)= \int_{\Sigma}x = 0
\]

Now if \( \Sigma \) has density \( m \) at \( p \), then its naturally weighted volume in the
Northern (and so in the Southern hemisphere) is at least \( m/2 \) times that of a
totally geodesic \( k \)-subsphere, which is
 \( \frac{m}{2}.\frac{\omega_{k-1}}{k} \). Moreover, by Lemma \ref{lem:comparison}, we recover
the following lower bound of unweighted volume, which was first proved by Cheng,
Li and Yau using the heat kernel. 

\begin{corollary}[cf. {\cite[Corollary 2]{Cheng.etal84_HeatEquationsMinimal}}]
Let \(\Sigma\) be a minimal \( k \)-submanifold of \( S^n \) that has no boundary in the geodesic
ball \(B(p,r)\) centred at \(p\) and of radius \(r \leq
\frac{\pi}{2}\). Suppose that \( \Sigma \) contains \( p \) with multiplicity \( m \).
Then the volume of \(\Sigma\cap B(p,r)\) is at least \(m\) times that of 
the \( k \)-ball of radius \(r\):
\[
     A(\Sigma\cap B(O,r)) \geq m \omega_{k-1}\int_{t=0}^{r} \sin^{k-1}(t)dt
   \]
\end{corollary}

\begin{proof}[Proof of Proposition \ref{prop:eps-sphere}]
Suppose that \( \Sigma \) is at distance \( \epsilon \) from \( p' \), which means \(
h_2 \geq a := 1 - \cos\epsilon \) on \( \Sigma \). On the Southern hemisphere, \( \Sigma \) is
\( (a, (U, 1)) \)-monotone with density \(\geq m \) at \( h_2 = 1 \). Lemma
\ref{lem:comparison} says that its \( (a,(1, \frac{1}{1-a})) \)-density on this hemisphere is at
least \( m \). So the volume \( A_- \) of \( \Sigma \) in this hemisphere can be bounded by
\[
A_-\geq  m.\omega_{k-1}\left[ \int_{a}^1 V(t)^{k/2 - 1} dt + \frac{1}{k}.\frac{1}{1-a} V(a)^{k/2}\right]
\]
where \( V(t) = 2t - t^2 \). On the Northern hemisphere, the unweighted density is at least
\( m \) and so the volume \( A_+ \) of \( \Sigma \) there is at least \( m
\frac{\omega_{k}}{2} \). Therefore
\[
A = A_+ + A_- \geq m \left[ \frac{\omega_k}{2} + \omega_{k-1}\left( \int_{a}^1(2t -
    t^2)^{k/2-1}dt + \frac{(2a - a^2)^{k/2}}{k(1-a)} \right) \right]
\]
So \( \epsilon \) satisfies
\begin{equation}
\label{eq:A-eps}
\frac{A}{m}\geq \frac{\omega_k}{2} + \omega_{k-1} \left[ \int_{\epsilon}^{\pi/2}\sin^{k-1}t dt + \frac{(1 - \cos^2\epsilon)^{k/2}}{k\cos\epsilon} \right].
\end{equation}
\end{proof}

One sees either from the domain of definition of \( \epsilon(\cdot,k) \) or from \eqref{eq:A-eps} that:
\begin{corollary}[]
  \label{cor:yau-conj}
If a closed minimal \( k \)-submanifold \( \Sigma \) of \( S^n \) has
  multiplicity \( m \) at a point, then its volume is at least \( m\omega_k \).
\end{corollary}

\begin{remark}[]
  \label{rem:yau-conj}
It was conjectured by Yau \cite[Problem 31]{Yau12_ChernGreatGeometer} that the volume of a
non-totally geodesic minimal hypersurface of \( S^n \) is lower bounded by that of the
Clifford tori. Corollary \ref{cor:yau-conj} shows that the volume of a non-embedded
minimal hypersurface of \( S^n \) is at least \( 2\omega_{n-1} \) and thus confirms the
conjecture for non-embedded hypersurfaces. By a different method, Ge and Li
\cite{Ge.Li22_VolumeGapMinimal} give a slightly weaker version of Corollary
\ref{cor:yau-conj}, which is also sufficient to confirm Yau's conjecture in this case.
\end{remark}

Proposition \ref{prop:eps-sphere} can be tested on the Veronese surface, which is
the image of the conformal harmonic immersion \( f: S^2 \longrightarrow S^4 \):
\[
f(x,y,z) = \left(\sqrt{3}xy, \sqrt{3}yz, \sqrt{3}zx, \frac{\sqrt{3}}{2} (x^2 - y^2),
\frac{1}{2}(x^2 + y^2 - 2z^2)\right).
\]
This map takes the same value on antipodal points of \( S^2 \) and descends to an
embedding of \( \mathbb{R}P^2 \) into \( S^4 \). Because \( f^*g_{S^4} = 3g_{S^2} \), the
image has area \(6\pi \). The antipodal point of \( f(0,0,1) \) is at distance
\( \epsilon = \frac{\pi}{3} \) to the surface. Because \( f \) is
\( SO(3) \)-equivariant, this is also the
smallest \( \epsilon \) such that the surface is \( \epsilon \)-antipodal.

\subsection{Minimising cones}

Area-minimising cones in Euclidean space appear naturally as \emph{oriented tangent cones}
of an area-minimising surface \cite{Morgan16_GeometricMeasureTheory}. It follows from
Lemma \ref{lem:less-than-tube} that in \(\mathbb{R}^n\) and \(\mathbb{H}^n\), sections of
a minimising cone cannot be spanned by a minimal submanifold.  By rescaling, a hyperbolic
minimising cone is also Euclidean minimising. The converse is also true, as pointed out by
Anderson \cite{Anderson82_CompleteMinimalVarieties}. More precisely, let \( \gamma \) be a
\( (k-1) \)-submanifold of a geodesic sphere centred at the origin of the Poincaré model
and suppose that the radial cone built on \( \gamma \) is Euclidean-minimising. Anderson
proved that the cone is also hyperbolic minimising.  The following result is an
improvement of this, in the sense that it also rules out complete minimal submanifolds of \( \mathbb{H}^n \).

\begin{proposition}
\label{prop:minimising-cone}
Let \( \gamma \) be a \( (k-1) \)-submanifold of the sphere at infinity (or a geodesic sphere) such that in the
Poincaré model, the radial cone built upon it is Euclidean minimising. Then there is no
minimal \( k \)-submanifold of \( \mathbb{H}^n \) with ideal boundary (respectively boundary) \( \gamma \).
\end{proposition}
\begin{proof}
If there was such a a minimal submanifold, then it would satisfy the Euclidean
monotonicity and thus by Lemma \ref{lem:less-than-tube} would have Euclidean area smaller
than that of the cone.
\end{proof}

To illustrate Proposition \ref{prop:minimising-cone}, let \( \gamma_0 \) be the standard
Hopf link given by the intersection of the pair of planes \( zw=0 \) with the unit sphere
\( S^3 \) of \( \mathbb{R}^4\cong \mathbb{C}^2 \). Since the pair of planes is Euclidean
minimising, it follows from Proposition \ref{prop:minimising-cone} that it is the only
minimal surface of \( \mathbb{H}^4 \) filling \( \gamma_0 \). Now let  \(\gamma_\epsilon\) be
the perturbed Hopf links cut out by the complex curves \(zw=\epsilon\), \(\epsilon\in (0,
\frac{1}{2})\). Clearly the pairs of planes filling them are no longer
Euclidean minimising and so the proof of Proposition \ref{prop:minimising-cone} fails. It
is possible to write down explicitly a family of minimal annuli of \( \mathbb{H}^4 \)
filling the
\( \gamma_\epsilon \).

In fact, let \(g\) be radially conformally flat metric, that is
\(g = e^{2\varphi(\rho)} g_E\) where  \(\varphi\) is a function of \(\rho:=|z|^2 + |w|^2\),
we will point out a family \( M_{C,\varphi} \) of \( g \)-minimal
annuli that are invariant by the \({S}^1\) action
\begin{equation}
\label{eq:hopf-action}
(z,w)\mapsto(ze^{i\theta}, we^{-i\theta})
\end{equation}
If \(\mathbb{C}^2\) is identified with the
space of quaternions by \((z,w) \mapsto z + jw\), this action corresponds to
multiplication on the left by \(e^{i\theta}\). 
These surfaces are obtained by rotating a curve in the real plane \(\im z = \im w =
0\). Such curve is given by an equation \(zw = F(\rho)\) where \(F\) is a real function on \(\rho\). The minimal surface
equation is equivalent to the following second order ODE on \(F\)
\[
 \frac{X'}{X} - \frac{Y'}{Y} + \frac{1}{\rho} +
\frac{\varphi'}{2}\left[8+\rho\left(\frac{X^2}{Y^2}-4\right)\frac{F'}{F}
\right]=0,\quad\text{where }X=F-F'\rho,\quad Y = \frac{F'}{2}\sqrt{\rho^2 - 4F^2}.
\]
This can be reduced to a first order ODE using symmetry. When we rotate a
solution curve in the real plane by an angle \( \alpha \), the new
curve is still a solution because this rotation is equivalent to multiplying on the right of \(z_1+jz_2\) by \(e^{j\alpha}\)
and so commutes with the left multiplication by \(e^{i\theta}\).

Concretely, by a change of variable \(F=\frac{\rho}{2}\sin\theta(\rho)\) the ODE reduces
either to the first order Bernoulli equation
\(\theta'^2 = -\rho^2 + C^2\rho^4 e^{4\varphi}\) for a parameter \(C > 0\), or to
\(\theta'=0\) which corresponds to pairs of 2-planes.  The profile curve can be described
in a more geometric fashion, as in Proposition \ref{prop:min-annuli}. Note that the curve
\eqref{eq:min-annuli} is a hyperbola when \( g=g_E \) and so the minimal surfaces
are the complex curves \(z_1 z_2 = \const\).

\begin{proposition}[]
\label{prop:min-annuli}
Let \(M_{C,\varphi}\) be the surface in \(\mathbb{C}^2\) given by rotating by
\eqref{eq:hopf-action} the following curve in \( \mathbb{R}^2 \):
\begin{equation}
\label{eq:min-annuli}
\sin^2\psi = C^2\frac{e^{-4\varphi}}{\rho^2}, \quad C > 0
\end{equation}
Here \(\psi\) is the angle formed by the tangent of the curve at a point \(p\) and the
radial direction \(\overrightarrow{Op}\). Then \(M_{C,\varphi}\) is
minimal under the metric \(g =
e^{2\varphi}g_E\). Up to \(SO(4)\), these annuli and the 2-planes are the only
minimal surfaces obtained as orbit of a real curve by the rotation \eqref{eq:hopf-action}.
\end{proposition}

When \( g \) is the hyperbolic metric, the renormalised area of the annuli \( M_C \) can
be computed to be
\begin{equation}
\label{eq:ARMC}
\mathcal{A}_R(M_C) = 4\pi \int_{a}^\infty\left[ \frac{\xi^2 - 1}{\sqrt{(\xi^2-1)^2 - C^2}} - 1\right] d\xi -
4\pi a
\end{equation}
where \( a =  (C+1)^{1/2}\). Note that just by Theorem
\ref{thm:upper-AR}, we know the renormalised area of \( M_C \) tends to \( -\infty \) as \(
C\to+\infty \).

Another radially conformally flat metric is the round sphere, with
\(e^\varphi = \frac{2}{1+\rho}\). The parameter \(C\) is in \((0,1)\). Seen from the
origin, the solution curve sweeps out an angle \(\theta_C\) between \(\frac{\pi}{2}\)
(\(C=0\), the surface is a totally geodesic \({S}^2\) ) and \(\frac{\pi}{\sqrt{2}}\)
(\(C=1\), the surface is (part of) the Clifford torus). In particular, if \(\theta_C\) is
a rational multiple of \(\pi\), we can close the surface by repeating the profile
curve. Benjamin Aslan pointed out to the author that the minimal annuli in this case are
bipolar transform of the Hsiang--Lawson annuli \(\tau_{0,1,\alpha} \). In
\cite{Hsiang.Lawson71_MinimalSubmanifoldsLow}, Hsiang and Lawson constructed a family
\(\tau_{p,q,\alpha} \) of minimal annuli in \( S^3 \) that are invariant by the
\( (p,q) \)-rotation
\[
(z,w) \longrightarrow (e^{ip\theta}z, e^{iq\theta}w)
\]
Here \( S^3 \) is seen as the unit sphere of \( \mathbb{C}^2 \). The bipolar transform was
defined by Lawson \cite{Lawson70_CompleteMinimalSurfaces} by wedging a conformal harmonic
map \(f: \Sigma \longrightarrow {S}^3\subset \mathbb{R}^4 \) with its \( \mathbb{R}^4
\)-valued Gauss map. The result is a map from \( \Sigma \) to the unit sphere
of \( \Lambda^2 \mathbb{R}^4 \cong \mathbb{R}^6 \), which is also conformal and
harmonic. This transforms a minimal surface of \({S}^3 \) into a minimal surface of
\({S}^5 \). If a surface is invariant by the \( (p,q) \)-rotation, its
bipolar transform is contained in a subsphere \( S^4 \) and is invariant by a
\( (p+q, p-q)\)-rotation. When \( p=0,q=1 \), they are the annuli \( M_C \).

\section{Weighted monotonicity in spaces with curvature bounded from above}
\label{sec:bounded-curv}
Fix a point \(O\) in a Riemannian manifold \((M^n,g)\), and let \(r_{\rm inj}\) be the
injectivity radius. Let \( r \) be the distance function to \(O\). Its Hessian at
a point \( p\in B(O,r_{\rm inj}) \) is:
\begin{equation}
\label{eq:hess-r}
\hess r (\partial_r, \cdot) =0, \qquad \hess r (v, v)=: I(v),\ \forall v\perp \partial_r
\end{equation}
where \(I(v)=\int_{\Gamma}\left(|\dot V|^2 - K_M(\dot\gamma,V)|V|^2\right)\) is the index
form of the Jacobi field \(V\) along the geodesic \(\Gamma\) between \(O\) and \(p\) that
interpolates \(0\) at \(O\) and \(v\) at \(p\). 
When the sectional curvature
satisfies \(K_M\leq -a^2\) (respectively \(b^2\)), one can check that
\(I(v) \geq a\coth(ar)|v|^2\) (respectively \(b\cot (br)|v|^2\)). This gives an estimate of \(\hess r\) on 
directions orthogonal to \(\partial_r\). 

\begin{proposition}[]
\label{prop:hess-r}
Inside \(B(O,r_{\rm inj})\), 
\begin{enumerate}
\item If \(K_M\leq -a^2\) then \(\hess (a^{-2}\cosh ar) \geq \cosh ar.\ g\).
\item If \(K\leq b^2\) then  \(\hess (-b^{-2}\cos br) \geq \cos br.\ g\) when \(r\leq \frac{\pi}{b}\).
\end{enumerate}
\end{proposition}
This means that the functions \(h=a^{-2}\cosh ar\) and \(h=-b^{-2}\cos br\) satisfy \(\hess h \geq
U.g\). 
Here \(U= a^2 h\) (respectively \(-b^2 h\)), \(V = |\nabla h|^2 = a^2
(h^2-1)\) (respectively \(-b^2 (h^2-1)\)) and we still have \(U = \frac{1}{2}V'\).
When \(M\) is \(\mathbb{H}^n\) or \({S}^n\), we recover the time-coordinate and
the Euclidean coordinate in Examples \ref{ex:xi-H} and \ref{ex:xi-S}.

As explained in Remark \ref{rem:monotonicity-warped}, we still have weighted monotonicity
theorem for \(h\).  It is more convenient here to see the weight as a function of \(r\)
instead of \(h\). The \emph{\(P\)-volume} functional of a \( k \)-submanifold is
\(A_{P}(\Sigma)(t):= \int_{\Sigma, r\leq t}P(r)\) and the \emph{\(P\)-density} is
\(\Theta_P:= \frac{A_P}{Q_P}\) where \( Q_P\) is the \(P\)-volume of a ball of radius
\(t\), not in \( M \) but in space-form:
\begin{equation}
\label{eq:Q-geodesic}
 Q_P(t):= \begin{cases}
\omega_{k-1}\int_{r\leq t}P(r)\frac{\sinh^{k-1} ar}{a^{k-1}}dr       ,  & \text{when } K_M\leq -a^2 \\
\omega_{k-1}\int_{r\leq t}P(r)\frac{\sin^{k-1} br}{b^{k-1}}dr       , & \text{when } K_M\leq b^2
       \end{cases}
     \end{equation}

The \emph{naturally weighted volume} and \emph{naturally weighted density} correspond to
\( P=U=\cosh ar \) (respectively \( \cos br \)), for which \( Q_U(t)=
\frac{\omega_{k-1}}{k} a^{-k}{\sinh^k at}  \) (\( \frac{\omega_{k-1}}{k}b^{-k} \sin^k bt \) respectively).
We define the \emph{eligible interval} \([0,r_{ \max})\) to be \([0, r_{\rm inj})\) when
\(K_M\leq -a^2\) and \([0, \min (r_{\rm inj}, \frac{\pi}{2b}))\) when \(K_M\leq b^2\).

\begin{theorem}[]
\label{thm:monotonicity-geodesic}
Let \(M\) be a Riemannian manifold with sectional curvature \(K_M\leq -a^2\) (or \(K_M\leq b^2\)) 
and \(\Sigma\subset M\) be an extension of a minimal \( k \)-submanifold by geodesic rays, then
the naturally weighted density \(\Theta_U(\Sigma)(t)\) is an increasing function in \( [0,
r_{\max})\). 
\end{theorem}

\begin{corollary}
Let \(\Sigma\subset M\) be a minimal \( k \)-submanifold containing the point \(O\) with
multiplicity \(m\) and \(l_t\) be the \( (k-1) \)-volume of the intersection \(
\Sigma\cap\{r=t\}\) under the metric \( g \) of \( M \).
For all \(t\) in the eligible
interval, one has\begin{equation}
\label{eq:bcurv-est}
Q_U(t) \leq A_U(\Sigma)(t)\leq \frac{l_t}{k}V(t)^{1/2}
\end{equation}
In particular,
\[
 l_t \geq \begin{cases}
 m\omega_{k-1}\left(\frac{\sinh at}{a}\right)^{k-1}	  ,  & \text{if $K_M\leq -a^2$} \\
 m\omega_{k-1}\left(\frac{\sin bt}{b}\right)^{k-1}	  , & \text{if $K_M\leq b^2$}
	  \end{cases}
\]
\end{corollary}
\begin{proof}
  The first half of \eqref{eq:bcurv-est} follows directly from Theorem
  \ref{thm:monotonicity-geodesic}. The second half is more subtle because the \(P\)-volume
  of a geodesic cone in \(M\) is no longer proportional to \( Q \) (so Lemma
  \ref{lem:less-than-tube} does not generalise). Instead, the upper bound of \(A_U\)
  follows from \eqref{eq:stokes}:
  \( A_U(\Sigma)(t) \leq \frac{1}{k}\int_{\Sigma, r=t}|\nabla^\Sigma h| \leq
  \frac{V(t)^{1/2}}{k}l_t.  \) The function \( V \) can also be rewritten as
  \( V(r) = a^{-2}\sinh^2ar \) ( \( b^{-2}\sin^2 br \) respectively).
\end{proof}

With an identical proof as Lemma \ref{lem:comparison}, we have:
\begin{lemma}[Comparison]
\label{lem:comparison-geodesic}
Let \(\Sigma\) be any \( k \)-submanifold of \( M \) not necessarily minimal and \(P_1,P_2\) be two
non-negative, continuous weights. Let \(Q_1, Q_2\) be defined from \(P_1,
P_2\) as in \eqref{eq:Q-geodesic} and  \( \Theta_1,\Theta_2
\) be the two densities. In the eligible interval, suppose that \( \Theta_2 \) is
increasing, then: 
\begin{enumerate}
\item If \(P_1\) is weaker than \(P_2\), i.e. \(\frac{P_1}{Q_1} \leq
   \frac{P_2}{Q_2}\), then \(\Theta_1 \) is also increasing and \(\Theta_1\geq\Theta_2\).
\item If \(P_2\) is weaker than \(P_1\), then \(\Theta_1\geq\Theta_2\).
\end{enumerate}
\end{lemma}

Note that to compare weights, it is necessary to mention the curvature bound \(a\) or
\(b\).
\begin{lemma}[]
\label{lem:compare-geodesic}
For any \(a, b\geq 0\) and \(u\geq v \geq 0\),
\begin{enumerate}
\item \(P_1 = \cosh v r\) is weaker than \(P_2 = \cosh ur\) when \(K_M\leq -a^2\),
\item \(P_1 = \cos ur\) is weaker than \(P_2 = \cos vr\) on the interval \([0. \frac{\pi}{2u}]\) when \(K_M\leq b^2\).
\end{enumerate}
\end{lemma}
Lemma \ref{lem:compare-geodesic} can be seen as a continuous version of the chain
\eqref{eq:chain}. In particular, when \(K_M\leq -a^2\) the monotonicity theorem holds for
any weight \(P_u =\cosh ur\) with \(u\in [0, a)\), including the uniform weight. When
\(K_M\leq b^2\), the monotonicity theorem holds for any weight \(P_u = \cos ur\) with
\(u\in[b,\infty)\). The second part of Lemma \ref{lem:compare-geodesic} can be used
to obtain a lower bound of volume in this case.

\begin{proposition}[]
\label{prop:uniform-area-sphere}
Suppose that \(K_M\leq b^2\). Let \( \Sigma \) be a minimal \( k \)-submanifold with
multiplicity \( m \) at the point \(O\) and no boundary in the interior of the ball
\(B(O,t)\) of radius \(t<r_{\max}\). Then
\[
A(\Sigma\cap B(O,t)) \geq m \omega_{k-1}\int_{r=0}^{t} \frac{\sin^{k-1}(br)}{b^{k-1}}dr.
\]
\end{proposition}
In particular, if \(M\) is simply connected, with curvature pinched between
\(\frac{b^2}{4}\) and \(b^2\) and \(\Sigma\subset M\) is a closed minimal submanifold,
then \( A(\Sigma)\geq \frac{1}{2}\omega_k b^{-k}\). A weaker version of this, with
\(\frac{1}{2}\omega_k\) replaced by the volume of the unit \(k\)-ball, was proved in
\cite{Hoffman.Spruck74_SobolevIsoperimetricInequalities}.

\bibliographystyle{amsplain}
\bibliography{GeoDiff}

\end{document}